\documentclass{siamart190516}

\usepackage{lipsum}
\usepackage{amsfonts}
\usepackage{graphicx}
\usepackage{epstopdf}
\usepackage{algorithmic}
\ifpdf
  \DeclareGraphicsExtensions{.eps,.pdf,.png,.jpg}
\else
  \DeclareGraphicsExtensions{.eps}
\fi
\usepackage{lipsum}
\usepackage{graphicx}
\usepackage{epstopdf}
\usepackage{algorithmic}
\usepackage{calligra}
\usepackage{mathtools}
\usepackage{hyperref}
\usepackage{autonum}
\usepackage{hhline}
\usepackage{array}
\usepackage{diagbox}
\usepackage{tcolorbox}
\usepackage{mdframed}
\usepackage{multicol}
\usepackage{graphicx}
\usepackage{subcaption}
\usepackage{moreverb}
\usepackage{bbm}
\usepackage{todonotes}
\usepackage{scalerel,amssymb}
\allowdisplaybreaks
\usepackage{mathrsfs}  
\usepackage{lineno}
\usepackage{todonotes}
\usepackage{tikz}
\usepackage{pgfplots}
\usetikzlibrary{arrows.meta}
\usepackage{siunitx}
\usepackage[numbers,sort&compress]{natbib}
\definecolor{mygreen}{HTML}{43a047}
\usepackage{subcaption}
\usepackage{doi}

\def\ulal{\underline{\alpha}}
\def\olal{\overline{\alpha}}
\def\op{\overline{p}}
\def\opsi{\overline{\psi}}
\def\oq{\overline{q}}
\newcommand{\Om}{\Omega}
\newcommand{\D}{\Delta}

\newcommand{\pt}{p_t}

\newcommand{\ptt}{p_{tt}}

\newcommand{\ddt}{\frac{\textup{d}}{\textup{d}t}}

\newcommand{\ds}{\, \textup{d} s }

\newcommand{\dxs}{\, \textup{d}x\textup{d}s}

\newcommand{\intTO}{\int_0^t \int_{\Omega}}

\newcommand{\eq}[1]{$#1$}
\newcommand{\nLtwo}[1]{\|#1\|_{L^2}}

\newcommand{\nLfour}[1]{\|#1\|_{L^4}}

\newcommand{\nHtwo}[1]{\|#1\|_{H^2}}
\newcommand{\nLtwoLtwo}[1]{\|#1\|_{L^2 (L^2)}}
\newcommand{\nLtwoLfour}[1]{\|#1\|_{L^2 (L^4)}}

\newcommand{\nLtwoLinf}[1]{\|#1\|_{L^2 (L^\infty)}}
\newcommand{\nLinfLtwo}[1]{\|#1\|_{L^\infty (L^2)}}

\newcommand{\nLinfLfour}[1]{\|#1\|_{L^\infty (L^4)}}
\newcommand{\nLinfLinf}[1]{\|#1\|_{L^\infty (L^\infty)}}

\newcommand{\prodLtwo}[2]{(#1, #2)_{L^2}}
\newcommand{\R}{\mathbb{R}} 
 
\newcommand{\Ltwo}{L^2(\Omega)}

\newcommand{\Honetwo}{{H_\diamondsuit^2(\Omega)}}
\newcommand{\Honethree}{{H_\diamondsuit^3(\Omega)}}

\newcommand{\spaceW}{X^{\textup{W}}}
\newcommand{\spaceK}{X^{\textup{K}}}
\newcommand{\LtwoLtwo}{L^2(0,T; L^2(\Omega))}
\newcommand{\LinfLinf}{L^\infty(0,T; L^\infty(\Omega))}
\newcommand{\CHone}{C_{H^1, L^4}}

\newcommand{\CHtwo}{C_{H^2, L^\infty}}

\newcommand{\TK}{\mathcal{T}^{\textup{K}}} 
\newcommand{\TW}{\mathcal{T}^{\textup{W}}}

\definecolor{grey}{rgb}{0.5,0.5,0.5}

\newsiamremark{remark}{Remark}
\newsiamremark{hypothesis}{Hypothesis}
\crefname{hypothesis}{Hypothesis}{Hypotheses}
\newsiamthm{claim}{Claim}

\headers{Inviscid limit of third-order acoustic equations}{B. Kaltenbacher and V. Nikoli\'c}

\renewcommand\footnotemark{}
\title{The inviscid limit of third-order linear and nonlinear acoustic equations\thanks{\vspace*{-2mm}\funding{The work of the first author was supported by the Austrian Science Fund {\sc fwf} under the grants P30054 and DOC 78.}}}

\author{Barbara Kaltenbacher\thanks{Department of Mathematics, Alpen-Adria-Universit\"at Klagenfurt, Austria (\email{barbara.kaltenbacher@aau.at}).}
\and Vanja Nikoli\'c\thanks{Department of Mathematics, Radboud University, The Netherlands  (\email{vanja.nikolic@ru.nl}).}
}

\usepackage{amsopn}

\ifpdf
\hypersetup{
  pdftitle={The inviscid limit of third-order linear and nonlinear acoustic equations},
  pdfauthor={B. Kaltenbacher and V. Nikolic}
}
\fi

\begin{document}

\maketitle

\begin{abstract}
We analyze the behavior of third-order in time linear and nonlinear sound waves in thermally relaxing fluids and gases as the sound diffusivity vanishes. The nonlinear acoustic propagation is modeled by the Jordan--Moore--Gibson--Thompson equation both in its Westervelt and in its Kuznetsov-type forms, that is, including quadratic gradient nonlinearities. As it turns out, sufficiently smooth solutions of these equations converge in the energy norm to the solutions of the corresponding inviscid models at a linear rate. Numerical experiments illustrate our theoretical findings.
\end{abstract}



\section{Introduction}
The present work focuses on the limiting behavior of linear and nonlinear equations that describe the motion of sound waves through thermally relaxing media, as the diffusivity of sound vanishes. In modeling, the need to combine thermal relaxation with the nonlinear and dissipative effects leads to third-order in time equations. They have the following general form:
\begin{equation}\label{generalJMGT}
\tau u_{ttt} +  u_{tt} - (\delta +\tau c^2) \D u_t -c^2 \D u = f(u, u_t, u_{tt}, \nabla u, \nabla u_t),
\end{equation}
where the function $u=u(x,t)$ may denote the acoustic pressure or acoustic velocity potential. The parameter $\tau>0$ denotes the thermal relaxation time, $c>0$ the speed of sound, and the coefficient $\delta$ is often referred to as the diffusivity of sound~\cite[\S 3]{hamilton1998nonlinear}. The parameter $\delta$ is relatively small in fluids and gases, which motivates our research into the behavior of solutions of \eqref{generalJMGT} as $\delta \rightarrow 0^+$. We take into account two types of nonlinearities, which have a physical motivation:
	\begin{align} \label{f1}
	f(u, u_t, u_{tt})= \frac12 (k u^2)_{tt}= k u u_{tt}+k u_t^2,
	\end{align}
	and
	\begin{align} \label{f2}
	f(u_t, u_{tt}, \nabla u, \nabla u_t)= \frac12(\kappa u_t^2+\sigma |\nabla u|^2 )_t= \kappa u_t u_{tt}+\sigma \nabla u \cdot \nabla u_t.
	\end{align}
\indent Third-order models of nonlinear acoustics in the form of \eqref{generalJMGT} originate from using a general temperature law within the governing system of equations, which includes conservation laws and constitutive equations of the medium~\cite{jordan2014second}. By the standard Fourier temperature law, a thermal disturbance at one point has an instantaneous effect elsewhere in the medium~\cite{liu1979instantaneous,gurtin1968general}, which may lead to an infinite speed of propagation paradox in waves. The Maxwell--Cattaneo law, on the other hand, introduces a time lag between the temperature gradient and the heat flux induced by it~\cite{kulish2003relationship,straughan2011heat} via
\[\tau \mathbf{q}_t+\mathbf{q}=- K\nabla \mathbf{\theta}.\]
The heat then propagates in time via thermal waves, which is often referred to as the second-sound phenomenon in the literature~\cite{jordan2014second, straughan2011heat}. Employing the Maxwell--Cattaneo temperature law within the governing equations of sound propagation leads to third-order in time models \eqref{generalJMGT} that avoid the paradox of infinite speed of propagation and instead have the expected hyperbolic character. \\
\indent Third-order wave motion may also originate from the presence of molecular relaxation when the pressure-density relation of the medium is not satisfied exactly, but up to a memory term; see~\cite[\S 1.1]{naugolnykh2000nonlinear} and~\cite[\S 4]{rudenko1977theoretical}. Such relaxation mechanisms typically occur in media with ``impurities"; these can be, for example, water with micro-bubbles, seawater, chemically reacting fluids, or a mixture of gases. As such, they arise in various applications. For instance, micro-bubbles are often used as a contrast agent in ultrasonic imaging~\cite{dijkmans2004microbubbles}. They are also known to increase the speed and efficacy of the focused ultrasound treatments~\cite{stride2010cavitation}.  \\
\indent In this work, we investigate the convergence of solutions to equations \eqref{generalJMGT} as $\delta \rightarrow 0^+$. The analysis is performed for $\tau>0$ fixed, and so the term ``inviscid limit" should be understood in this context as the vanishing sound diffusivity limit in thermally relaxing fluids or gases. We consider the limiting behavior in the linearized versions of \eqref{generalJMGT}  on smooth bounded domains as well as in the nonlinear PDE, which is referred to as the Jordan--Moore--Gibson--Thompson equation in the literature. \\
\indent Our main results pertain to the convergence of solutions to \eqref{generalJMGT} with homogeneous Dirichlet boundary conditions in the energy norm to their inviscid counterparts, as the sound diffusivity vanishes. It turns out that sufficiently smooth solutions of \eqref{generalJMGT} converge to the solutions of the inviscid equation in the energy norm at a linear rate:
\[
\|u^{(\delta)}-u\|_{\textup{E}} \lesssim \delta \qquad \text{as } \ \delta \rightarrow 0^+;
\]
see also \eqref{spaceE} below. The smoothness requirements are naturally higher in the presence of quadratic gradient nonlinearity~\eqref{f2}. We refer to upcoming Theorems~\ref{Thm:West_WeakLimit_lin}, \ref{Thm:West_WeakLimit}, and~\ref{Thm:Kuzn_Limit} for details.\\
\indent We organize the remaining of our exposition as follows. In Section~\ref{Sec:Modeling}, we provide more details about the mathematical modeling of ultrasonic waves in thermally relaxing fluids and give an overview of related results. Section~\ref{Sec:Lin_JMGTWest} deals with a linear version of \eqref{generalJMGT}. There we derive a uniform bound with respect to $\delta$ for its solutions and prove a limiting result as $\delta$ vanishes; see Theorem~\ref{Thm:West_WeakLimit_lin}. Section~\ref{Sec:FixedPoint_JMGTWest_Limit} extends the investigation to the JMGT equation with the right-hand side nonlinearity given by \eqref{f1} by employing Banach's Fixed-point theorem; see Theorem~\ref{Thm:West_WeakLimit}. In Section \ref{Sec:Lin_JMGTKuzn}, we return to a linearized problem to derive a higher-order energy bound that is uniform with respect to $\delta$. This is needed so that in Section~\ref{Sec:FixedPoint_KuznWest_Limit} we can analyze the limiting behavior of the JMGT equation with the right-hand side given by \eqref{f2}; see Theorems~\ref{Thm:FixedPoint_JMGTKuzn} and~\ref{Thm:Kuzn_Limit}. Finally, in Section~\ref{Sec:NumResults} we provide the results of numerical experiments, which illustrate our theory and the convergence rate $O(\delta)$.
\section{Mathematical modeling and related work  } \label{Sec:Modeling}
The Jordan--Moore--Gibson--Thompson (JMGT) equation, given by
\begin{equation}\label{JMGT_Kuzn}
\tau \psi_{ttt} + \psi_{tt} - (\delta +\tau c^2) \D \psi_t -c^2 \D \psi = \frac12 (\kappa \psi_t^2+\sigma |\nabla \psi|^2 )_t, 
\end{equation}
is a model of ultrasonic propagation that accounts for thermal relaxation, dissipation due to viscosity, and nonlinear effects of quadratic type. It is expressed in terms of the acoustic velocity potential $\psi=\psi(x,t)$; we refer to~\cite{jordan2014second} for its derivation. The parameter $\tau>0$ represents the relaxation time of the heat flux and the coefficient $c>0$ stands for the speed of sound.  The nonlinear parameters on the right-hand side are typically  $\sigma=2$ and $\kappa= \frac{B/A}{c^2}$, where $B/A$ is the parameter of nonlinearity that arises from the pressure-density relation in a given medium. For the purposes of our analysis, it is sufficient to assume that $\sigma$, $\kappa \in \mathbb{R}$. \\
\indent Local nonlinear effects in sound propagation can often be neglected if the propagation distance is sufficiently large in terms of the number of wavelengths. In such cases, employing the approximation $|\nabla \psi|^2 \approx (1/c^2) \, \psi_t^2$ in \eqref{JMGT_Kuzn} leads to
\begin{equation}  \label{JMGT_West_potential}
\begin{aligned}
&\tau \psi_{ttt}+ \psi_{tt}-c^2\Delta \psi-(\delta+\tau c^2) \Delta \psi_t
=\frac{1}{c^2}\left(1+\frac{B}{2A}\right) (\psi_{t}^{2})_t.
\end{aligned}
\end{equation}
This equation is frequently expressed in terms of the acoustic pressure. Formally differentiating \eqref{JMGT_West_potential} with respect to time and then using the relation $p=\varrho \psi_t$, where $\varrho$ is the medium density, yields
\begin{equation}\label{JMGT_West}
\tau p_{ttt} + p_{tt} - (\delta +\tau c^2) \D p_t -c^2\D p = \frac12 (k p^2)_{tt},
\end{equation}
where now $k=\frac{1}{\varrho c^2}\left(1+\frac{B}{2A}\right)$. Again, for our theoretical purposes, we can take $k \in \mathbb{R}$. Letting $\tau \rightarrow 0$ in \eqref{JMGT_Kuzn} and \eqref{JMGT_West} leads to the classical second-order models of nonlinear acoustics -- Kuznetsov's~\cite{kuznetsov1971equations} and Westervelt's~\cite{westervelt1963parametric} equations, respectively. For this reason and to distinguish different third-order equations, we will henceforth refer to \eqref{JMGT_Kuzn} and \eqref{JMGT_West} as the JMGT--Kuznetsov and JMGT--Westervelt equation, respectively.   \\
\indent A linear version of these equations is referred to as the Moore--Gibson--Thompson (MGT) equation~\cite{moore1960propagation,thompson}:
\begin{equation}\label{MGT}
\tau p_{ttt} + \alpha p_{tt} - (\delta +\tau c^2) \D p_t -c^2\D p = 0.
\end{equation}
\indent Early mathematical investigations of third-order in time linear PDEs include the studies on classical solutions and Cauchy problems by Varlamov in~\cite{varlamov2003time, varlamov1987fundamental, varlamov1992hyperbolic}, and Renno in~\cite{renno1984wave}. Matkowsky and Reiss studied the asymptotic expansion of solutions to the linear problem as $\tau \rightarrow 0^+$ in~\cite{matkowsky1971asymptotic}. Further studies on the theory of singular perturbation can be found in~\cite{varlamov1988asymptotic, anna1980three}. In~\cite{varlamov1990asymptotic}, Varlamov and Nesterov analyzed the linear equation with spatially varying coefficients, establishing results on the existence and uniqueness of classical solutions and their asymptotic expansion as $\tau \rightarrow 0^+$. Nonlinear third-order propagation was seemingly first investigated by Varlamov in~\cite{varlamov2000long, varlamov1997third} in terms of existence and asymptotic behavior of classical solutions with a heuristically motivated right-hand side nonlinearity given by $\Delta (u^2)$.  \\
\indent A more recent mathematical research on third-order ultrasonic waves was initiated in~\cite{kaltenbacher2011wellposedness} with a semigroup approach employed in the well-posedness and stability analysis of the linear equation \eqref{MGT}. As concluded in~\cite{kaltenbacher2011wellposedness}, exponential stability of solutions requires that \eq{\gamma:=\alpha- \frac{\tau c^2}{\delta+\tau c^2}>0.}
The problem is unstable when $\gamma<0$ and marginally stable when $\gamma=0$. We note that in this work the focus is on the short-time behavior, and therefore no assumptions on the sign of $\gamma$ will be made. \\
\indent Since the results of~\cite{kaltenbacher2011wellposedness}, the interest in the qualitative behavior of third-order acoustic equations has flourished significantly, and these linear and nonlinear models represent by now a very active area of research. We thus provide the reader with a selection of relevant works here on the analysis of linear~\cite{pellicer2019wellposedness, bucci2019regularity, dell2017moore, conejero2015chaotic, pellicer2020uniqueness, marchand2012abstract} and nonlinear third-order acoustic equations~\cite{kaltenbacher2012well, racke2020global, chen2019nonexistence}. We also note that the limiting behavior of solutions for vanishing thermal relaxation time $\tau$ (and fixed $\delta>0$) has been studied in, e.g.,~\cite{KaltenbacherNikolic, bongarti2020vanishing}. Taking into account the effects of both thermal and molecular relaxation leads to third-order equations with memory, which have also been a topic of extensive research recently; see, for example,~\cite{lasiecka2017global, dell2016moore, alves2018moore} and the references given therein. To the best of our knowledge, the present work is the first dealing with the convergence of solutions to third-order equations as the sound diffusivity $\delta$ vanishes. 
\subsection{Notation and auxiliary results}
Before proceeding, let us briefly set the notation and recall commonly used inequalities and embedding results. To simplify notation, we often omit the spatial domain and the time interval when writing norms; in other words, $\|\cdot\|_{L^p (L^q)}$ denotes the norm on $L^p(0,T;L^q(\Omega))$.\\
\indent Throughout the paper, we make the following regularity assumption concerning the spatial domain:
\begin{itemize}
	\item [($A_1$)] $\Om \subset \R^d$ is an open, bounded, and $C^{1,1}$ regular \emph{or} polygonal and convex set, where $d \in \{1, 2, 3\}$. 
\end{itemize}
When stating solution spaces for $p$ and $\psi$, we denote
\begin{equation} \label{sobolev_withtraces}
\begin{aligned}
\Honetwo=H_0^1(\Omega)\cap H^2(\Omega),\quad
\Honethree=\left\{\psi \in H^3(\Omega)\,:\, \mbox{tr}_{\partial\Omega} \psi = 0, \  \mbox{tr}_{\partial\Omega} \D \psi = 0\right\}.
\end{aligned}
\end{equation}
In the analysis, we will need to rely on the boundedness of the operator $(-\Delta)^{-1}:L^2(\Omega)\to \Honetwo$. We point out that since we do not need a stronger elliptic regularity result than this, we do not introduce stronger regularity assumptions than given in $(A_1)$ on $\Omega$. Additionally, we will often use the Poincar\'{e}--Friedrichs inequality and the continuous embeddings $H^1(\Omega)\hookrightarrow L^4(\Omega)$ and $H^2(\Omega) \hookrightarrow L^\infty(\Omega)$.\\
\indent We occasionally use $x \lesssim y$ to denote $x \leq C y$, where the generic constant $C>0$ does not depend on the sound diffusivity $\delta$, but may depend on other medium parameters and the final time $T$.
\section{The generalized Moore--Gibson--Thompson equation}  \label{Sec:Lin_JMGTWest}
We begin by investigating an initial boundary-value problem for the following generalization of the Moore--Gibson--Thompson equation:
\begin{subequations} \label{IBVP_LinWestJMGT}
\begin{equation}\label{LinWestJMGT}
\tau p_{ttt} + \alpha p_{tt} - (\delta +\tau c^2) \D p_t -c^2\D p - \mu p_t - \eta p= f, 
\end{equation}
\text{with $p_{\vert \partial \Om}=0$ and initial conditions}
\begin{align} \label{IC:LinWestJMGT}
(p, \pt, \ptt)\vert_{t=0}=(p_0, p_1, p_2).
\end{align}
\end{subequations}
In particular, we wish to derive an energy bound for \eqref{LinWestJMGT} that is uniform with respect to $\delta$ and will later allow us to derive the corresponding bound for the nonlinear JMGT--Westervelt equation. For this reason, the coefficients $\alpha$, $\mu$, and $\eta$ are space and time dependent. We note that the standard Moore--Gibson--Thompson equation \eqref{MGT} is obtained as a special case of \eqref{LinWestJMGT} by setting $\mu$ and $\eta$ to zero. To facilitate the analysis, we make the following regularity assumptions. \vspace*{1mm}
\begin{itemize}
	\item[($A_2$)] 
	The coefficients $\alpha$, $\mu$, and $\eta$ are sufficiently smooth and  uniformly bounded:
	\begin{equation}
\begin{aligned}
	&1-\alpha\in W^{1,\infty}(0,T;\Honetwo),\
	\mu\in L^\infty(0,T;\Honetwo),\
	\eta\in L^\infty(0,T;H_0^1(\Omega)),\\[1mm]
	&\|1-\alpha\|_{W^{1,\infty}(\Honetwo)},\ \|\mu\|_{L^\infty(\Honetwo)},\ \nLinfLtwo{\nabla\eta}\leq R,
\end{aligned}
\end{equation} 
where $R$ is a positive constant independent of $\delta$. This assumption further implies that	\eq{\ulal \leq \alpha \leq \olal} for $\ulal=1-\CHtwo R$ and $\olal=1+\CHtwo R$.    \vspace*{1mm}
	\item[($A_3$)] The initial conditions \eqref{IC:LinWestJMGT} satisfy
	$(p_0, p_1, p_2) \in \spaceW_0=\Honetwo\times\Honetwo\times H_0^1(\Omega).$ \vspace*{1mm}
	\item[($A_4$)] The source term satisfies \eq{f \in L^2(0,T; H_0^1(\Omega)) \cap L^\infty(0,T ;L^2(\Omega)).} \vspace*{1mm}
\end{itemize}
Note that since possibly $\ulal<0$, we do not make a non-degeneracy assumption on the coefficient $\alpha$ here. As already mentioned, this stems from the fact that we are interested in the short-time behavior of solutions. For an estimate in weaker norms, we will replace ($A_3$) and ($A_4$) by the following weaker assumptions. \vspace*{1mm}
\begin{itemize}
	\item[$\widetilde{(A_3)}$] The initial conditions \eqref{IC:LinWestJMGT} satisfy
	$(p_0, p_1, p_2) \in H_0^1(\Omega)\times H_0^1(\Omega)\times\Ltwo.$ \vspace*{1mm}
	\item[$\widetilde{(A_4)}$] The source term satisfies \eq{f \in \LtwoLtwo.} \vspace*{1mm}
\end{itemize}
We next prove a well-posedness result with a uniform bound in $\delta$ for \eqref{LinWestJMGT}. Since we are interested in the limit as $\delta\to0^+$, we may restrict our attention to a bounded interval $[0,\overline{\delta}]$ for some $\overline{\delta}>0$ without loss of generality.
\begin{proposition}\label{propLinWest}
Let assumptions \textup{($A_1$)--($A_4$)} hold and let $\tau$, $c$, $\overline{\delta}>0$. Then for $\delta \in [0, \overline{\delta}]$, the initial boundary-value problem \eqref{IBVP_LinWestJMGT} has a unique solution
\begin{equation}\label{regularityWest}
p \in \spaceW= W^{3,\infty}(0,T;\Ltwo)\cap W^{2,\infty}(0,T;H_0^1(\Omega))\cap W^{1,\infty}(0,T;\Honetwo).
\end{equation}	
Furthermore, this solution satisfies
\begin{equation}\label{enest1_linwest}
\begin{aligned}
&\nLtwo{\nabla p_{tt}(t)}^2+\nLtwo{\D p_t(t)}^2+\nLtwo{\D p(t)}^2 \\
\leq&\,\begin{multlined}[t] C(\tau)e^{K(\tau)(R^2+1)T}
\left(\nLtwo{\nabla p_2}^2 + \nLtwo{\D p_1}^2  + \nLtwo{\D p_0}^2 \right.  \left. +\nLtwoLtwo{\nabla f}^2\right), \end{multlined}
\end{aligned}
\end{equation}	
where the constants $C(\tau)$ and $K(\tau)$ tend to $+\infty$ as $\tau\to0^+$, but are independent of $R$, the final time $T$, and the sound diffusivity $\delta$.\\
\indent If instead of assumptions \textup{($A_3$)} and \textup{($A_4$)}, the weaker \textup{$\widetilde{(A3)}$} and \textup{$\widetilde{(A_4)}$} hold, any solution $p$ of \eqref{IBVP_LinWestJMGT} (if it exists) satisfies the estimate 
\begin{equation}\label{enest0_linwest}
\begin{aligned}
&\nLtwo{p_{tt}(t)}^2+\nLtwo{\nabla p_t(t)}^2+\nLtwo{\nabla p(t)}^2 \\
\leq&\, \begin{multlined}[t] C(\tau)e^{K(\tau)(R^2+1)T}
\left(\nLtwo{p_2}^2 + \nLtwo{\nabla p_1}^2  + \nLtwo{\nabla p_0}^2 \right.  \left.
+\nLtwoLtwo{f}^2\right). \end{multlined}
\end{aligned}
\end{equation}	
\end{proposition}
\begin{proof}
 The proof can be conducted by employing smooth Faedo--Galerkin approximations in space. In particular, we can project the problem onto the span $V_n$ of the first $n$ eigenfunctions of the Dirichlet Laplacian pointwise in time; cf.~\cite{roubivcek2013nonlinear, evans2010partial}. We will focus here on deriving the crucial energy bound for the Galerkin approximations $p^n$ and refer to, for example,~\cite{KaltenbacherNikolic, kaltenbacher2020parabolic} for details regarding the application of the Faedo--Galerkin procedure in nonlinear acoustics. To ease the notation, we drop the superscript $n$ below and use just $p$. \\
 \indent In view of the fact that this linear equation is close to a wave equation for $z=\tau p_t+p$, we test it with $-\D (\tau p_t+p)_t$ and use integration by parts with respect to space. Noting that $p=\D p=0$ on $\partial \Om$ for sufficiently smooth Galerkin approximations, 
we can integrate by parts with respect to space to obtain
\begin{equation} \label{energy_id}
\begin{aligned}
&\begin{multlined}[t]\frac12 \tau^2\ddt \nLtwo{\nabla p_{tt}}^2
+\tau (\prodLtwo{\alpha\nabla p_{tt}}{\nabla p_{tt}}-\nLtwo{\nabla p_{tt}}^2)+ \frac12c^2 \ddt \nLtwo{\D p}^2
\\+ \frac12\tau(\delta+\tau c^2)\ddt \nLtwo{\D p_t}^2
+\frac12\ddt \prodLtwo{\alpha\nabla p_t}{\nabla p_t}
+\delta \nLtwo{\D p_t}^2 \end{multlined}\\
=&\,\begin{multlined}[t] \prodLtwo{f_1(p)+f_2(p)+\nabla f}{\nabla (\tau p_{tt}+p_t)} -\frac12 \prodLtwo{\alpha_t\nabla p_t}{\nabla p_t} \\
-\tau c^2 \ddt \prodLtwo{\D p}{\D p_t}
 -\tau \ddt \prodLtwo{\nabla p_{tt}}{\nabla p_t},\end{multlined}
\end{aligned}
\end{equation}
where 
$f_1(p)= p_t\,\nabla \mu +  \mu\, \nabla p_t + p_{tt}\,\nabla \alpha$ and $f_2(p)= p\nabla\eta+\eta\nabla p$ satisfy
\begin{equation} \label{est_f1f2}
\begin{aligned}
&\nLtwo{f_1(p)}\leq (\CHone^2+\CHtwo)\nLtwo{\D \mu}\nLtwo{\nabla p_t} 
+ \CHone^2\nLtwo{\D \alpha}\nLtwo{\nabla p_{tt}},\\
&\nLtwo{f_2(p)}\leq (\CHtwo +  \CHone^2)
\nLtwo{\nabla\eta}\nLtwo{\D p}.
\end{aligned}
\end{equation}
We next integrate \eqref{energy_id} with respect to time and use Young's inequality to arrive at the following energy estimate:
\begin{equation}
\begin{aligned}
&\begin{multlined}[t]\tau^2\frac{1-\epsilon}{2} \nLtwo{\nabla p_{tt}(t)}^2
+\frac12 \tau(\delta+\tau \frac{c^2}{2}) \nLtwo{\D p_t(t)}^2
+\delta \int_0^t \nLtwo{\D p_t}^2\ds 
+ \frac12c^2 \nLtwo{\D p(t)}^2 
\end{multlined}\\
\leq&\, \begin{multlined}[t] \frac12 \tau^2\nLtwo{\nabla p_{tt}(0)}^2
+ \frac12 \tau(\delta+\tau c^2) \nLtwo{\D p_t(0)}^2 
+\tau c^2 \prodLtwo{\D p(0)}{\D p_t(0)} \\
+\tau \prodLtwo{\nabla p_{tt}(0)}{\nabla p_t(0)}
+\prodLtwo{\alpha(0)\nabla p_t(0)}{\nabla p_t(0)}
+\frac12 c^2 \nLtwo{\D p(0)}^2\\
+\frac12\int_0^t\nLtwo{\nabla p_t}^2\ds
+\tau (\|\mu\|_{L^\infty(L^\infty)}+1) \int_0^t \nLtwo{\nabla p_{tt}}^2\ds 
+\frac12\left(\epsilon^{-1}-\ulal\right) \nLtwo{\nabla p_t(t)}^2\\
\hspace*{-1.1cm}+\int_0^t\left(\frac{\tau}{2}\nLtwo{f_1(p)+f_2(p)}^2 + \nLtwo{f_1(p)+f_2(p)-\frac12 \alpha_t\, \nabla p_t}^2\right)\ds
+\frac{\tau\!+\!2}{2}\int_0^t\nLtwo{\nabla f}^2\ds, \end{multlined}
\end{aligned}
\end{equation}
where $\epsilon \in (0,1)$.     mFurthermore, \eq{\nLtwo{\alpha_t\, \nabla p_t}\leq \CHtwo \nLtwo{\D \alpha_t}\nLtwo{\nabla p_t}} and 
\[\nLtwo{\nabla p_t(t)}^2=\nLtwo{\nabla p_t(0)+\int_0^t \nabla p_{tt}\ds}^2\leq 2\nLtwo{\nabla p_t(0)}^2 + 2t\int_0^t\nLtwo{\nabla p_{tt}}^2\ds.
\]
Gronwall's inequality therefore yields 
\begin{equation}
\begin{aligned}
&\tau^2\nLtwo{\nabla p_{tt}(t)}^2+\tau^2c^2\nLtwo{\D p_t(t)}^2+c^2\nLtwo{\D p(t)}^2 \\
\leq&\, Ce^{K(\tau)(R^2+1)T} \big(\nLtwo{\nabla p_{tt}(0)}^2 + \nLtwo{\D p_t(0)}^2  + \nLtwo{\D p(0)}^2
+\nLtwoLtwo{\nabla f}^2\big),
\end{aligned}
\end{equation}
with a constant $C>0$ that only depends on the medium parameters $\tau$, $c^2$, and the embedding constants $\CHone$ and $\CHtwo$, but is independent of $\delta\in[0,\bar{\delta}]$. This further yields \eqref{enest1_linwest}, at first in its semi-discrete version.

Additionally, from the (Galerkin-discretized) PDE, we obtain 
\begin{equation} \label{bound_pttt}
\begin{aligned}
\nLtwo{p_{ttt}(t)}^2 
\leq&\,\begin{multlined}[t] \frac{1}{\tau} \nLtwo{p_{ttt}(t)} \nLtwo{\alpha(t) p_{tt}(t) - (\delta +\tau c^2) \D p_t(t) -c^2\D p(t)\\ - \mu(t) p_t(t) - \eta(t) p(t) - f(t)}. \end{multlined}
\end{aligned}
\end{equation}
By the semi-discrete version of \eqref{enest1_linwest}, the above inequality implies that also
\begin{equation} \label{est_pttt}
\nLinfLtwo{p_{ttt}}\leq \begin{multlined}[t] \tilde{C}(T,R,\tau,c) 
\big(\nLtwo{\nabla p_{tt}(0)}^2 + \nLtwo{\D p_t(0)}^2  + \nLtwo{\D p(0)}^2
\big. \\ \big.+\nLtwoLtwo{\nabla f}^2  +\nLinfLtwo{f} \big) \end{multlined}
\end{equation}
and the PDE is satisfied in an $L^\infty(0,T;L^2(\Omega))$ sense. The obtained semi-discrete bounds carry over to the solution of \eqref{IBVP_LinWestJMGT} via standard compactness arguments; cf.~\cite[Theorem 3.1]{KaltenbacherNikolic} and~\cite[Proposition 3.1]{kaltenbacher2020parabolic}. We omit the details here. By~\cite[Lemma 3.3]{temam2012infinite}, it follows from $p \in \spaceW$ that
\begin{equation}
\begin{aligned}
p \in\, C([0,T]; \Honetwo),
\ p_t \in\, C_{w}([0,T]; \Honetwo),\
p_{tt} \in\, C_{w}([0,T]; H_0^1(\Om)),
\end{aligned}
\end{equation}
and thus initial conditions $p_1$ and $p_2$ are attained weakly. \\
\indent The second part of the proof is concerned with the bound \eqref{enest0_linwest} in the weaker norm, which can be obtained by testing with $(\tau p_t+p)_t$ in place of $-\D(\tau p_t+p)_t$. At first, this yields the energy identity
\begin{equation}
\begin{aligned}
&\begin{multlined}[t]\frac12 \tau^2\ddt \nLtwo{p_{tt}}^2
+\tau (\prodLtwo{\alpha p_{tt}}{p_{tt}}-\nLtwo{p_{tt}}^2)
+ \frac12\tau(\delta+\tau c^2)\ddt \nLtwo{\nabla p_t}^2\\
+\frac12\ddt \prodLtwo{\alpha p_t}{p_t}
+\delta \nLtwo{\nabla p_t}^2
+ \frac12c^2 \ddt \nLtwo{\nabla p}^2 \end{multlined}\\
=&\,\begin{multlined}[t] \prodLtwo{\mu p_t+f}{\tau p_{tt}+p_t} -\frac12 \prodLtwo{\alpha_t p_t}{p_t} 
-\tau c^2 \ddt \prodLtwo{\nabla p}{\nabla p_t}
 -\tau \ddt \prodLtwo{p_{tt}}{p_t}.\end{multlined}
\end{aligned}
\end{equation}
Usual computations with the the right-hand side terms then lead to \eqref{enest0_linwest}.  
\end{proof}
Introducing the energy of the solution at time $t \in [0,T]$ as
\begin{equation}\label{energy}
E[p](t)=\frac{\tau^2}{2}\nLtwo{p_{tt}(t)}^2+\frac{\tau^2 c^2}{2} \nLtwo{\nabla p_t(t)}^2 +\frac{c^2}{2} \nLtwo{\nabla p(t)}^2
\end{equation}
we see that on one hand, it dominates the physical energy for the auxiliary function $z=\tau p_t + p_t$:
\begin{equation}\label{Etil}
\tilde{E}[z](t)=\frac{1}{2}\nLtwo{z_t(t)}^2+\frac{c^2}{2} \nLtwo{\nabla z(t)}^2.
\end{equation} 
On the other hand, by \eqref{enest0_linwest}, it satisfies an energy estimate of the form
\begin{equation}
E[p](t)\leq C(T,R,\tau) \, E[p](0) + \int_0^t \nLtwo{f(s)}\ds.
\end{equation}
We will use $E[p]$ to establish the convergence rate as $\delta\to0^+$ in the space 
\begin{align} \label{spaceE}
\textup{E} = W^{2,\infty}(0,T;\Ltwo)\cap W^{1,\infty}(0,T;H_0^1(\Omega))
\end{align}
induced by this energy.
\begin{theorem} \label{Thm:West_WeakLimit_lin}
Under the conditions of Proposition~\ref{propLinWest}, the family of solutions $\{p^{(\delta)}\}_{\delta>0}$ to the generalized Moore--Gibson--Thompson equation converges in the topology induced by the energy norm for the wave equation to the solution $p$ of the inviscid equation as $\delta\to0^+$ at a linear rate. 
\end{theorem}
\begin{proof}
Let $\delta$, $\delta' \in [0, \overline{\delta}]$. Furthermore, let $p^{(\delta)}$ and $p^{(\delta')}$ be the solutions of \eqref{IBVP_LinWestJMGT} with the sound diffusivity $\delta$ and $\delta'$, respectively. We follow the general strategy of~\cite{kato1972nonstationary, tani2017mathematical, kaltenbacher2020parabolic, showalter1976regularization} by proving that $\{\op^\delta\}$ is a Cauchy sequence in suitable topology, where $\op=p^{(\delta)}-p^{(\delta')}$. We note that $\op$ solves the equation
\begin{equation} \label{WestJMGT_Cauchy_eq_lin}
\begin{aligned}
&\tau \op_{ttt}+\alpha\op_{tt}-(\delta+\tau c^2) \Delta \op_t-c^2 \Delta \op - \mu \op_t -\eta \op
= (\delta-\delta')\Delta p_t^{(\delta')}
\end{aligned}
\end{equation}
supplemented by zero initial conditions. 
Applying estimate \eqref{enest0_linwest} in Proposition~\ref{propLinWest} with $f=(\delta-\delta')\Delta p_t^{(\delta')}$ directly yields
\begin{equation}
E[p](t)\leq  C(\tau)\exp(K(\tau)(R^2+1)T) \, |\delta-\delta'|^2 \, \bar{R}^2
\end{equation}	
with energy $E[p]$ defined as in \eqref{energy}. 
Here we have used the bound 
\[
\begin{aligned}
\bar{R}^2=& \, C(\tau)e^{K(\tau)(R^2+1)T} \left(\nLtwo{\nabla p_2}^2 + \nLtwo{\D p_1}^2  + \nLtwo{\D p_0}^2+\nLtwoLtwo{\nabla f}^2\right)
\end{aligned}
\] 
on $\nLtwoLtwo{\Delta p_t^{(\delta')}}$, resulting from \eqref{enest1_linwest}. The stated rate is obtained by setting $\delta'$ to zero.
\end{proof}
\begin{remark}[Perturbation of the wave speed] \label{Remark:PerturbedSpeed}
Regularizing perturbations of the second-order linear wave equations have been studied thoroughly in the literature; see, for example,~\cite[\S \,8.5]{lions2012non} and~\cite{showalter1976regularization}. In view of the fact that \eqref{LinWestJMGT} can be seen as a second-order wave equation for $z=\tau p_t+p$, we compare these known results to Theorem~\ref{Thm:West_WeakLimit_lin} above.
To this end, for simplicity, we set $\alpha=1$, $\mu=\eta=f=0$, which yields
\begin{equation}\label{zwave}
z_{tt}-c^2\D z - \frac{\delta}{\tau} \D z +\frac{\delta}{\tau}\D p=0\,.
\end{equation}
It thus becomes apparent that our results are not covered by those in~\cite{showalter1976regularization, lions2012non}, which involve parabolic and non-singular perturbations; cf.~\cite[Theorem 8.3]{lions2012non} and ~\cite[Propositions 7 and 9]{showalter1976regularization}. In fact, equation \eqref{zwave} reveals that we do not deal with damping, but rather with a perturbation of the wave speed $c$.
\end{remark}
\begin{remark}[Heterogeneous media]
When besides $\alpha$, $\mu$, $\eta$, also $\tau$, $\delta$, $k$, and $c$ are space (and possibly time) dependent coefficients, all statements in this section remain valid, as long as these coefficients are sufficiently smooth and $c$ and $\tau$ are bounded away from zero. In particular, an inspection of the proofs shows that it suffices to have an $L^\infty$ sound speed satisfying \eq{0<\underline{c}\leq c(x)\leq \overline{c}} for reproducing the higher-order energy estimate \eqref{enest1_linwest}. This allows for the practically relevant setting of piecewise-constant sound speed. Note that for obtaining the lower-order energy estimate \eqref{enest0_linwest}, integration by parts with respect to space requires existence and a certain integrability of the gradient of $c^2$, though. 
\end{remark}

\section{Uniform bounds for the JMGT--Westervelt equation and the inviscid limit} \label{Sec:FixedPoint_JMGTWest_Limit}
We next wish to extend the study of the limiting behavior of the solutions to the nonlinear JMGT--Westervelt equation given by \eqref{JMGT_West}. To derive $\delta$-independent bounds on solutions to \eqref{JMGT_West}, we employ a fixed-point argument on the mapping \eq{\TW:q\mapsto p,} which associates $q$ with the solution of the linearized equation
\begin{equation}
\tau p_{ttt} + \alpha p_{tt} - (\delta +\tau c^2) \D p_t -c^2\D p - k q_t p_t = 0
\end{equation}
with homogeneous Dirichlet data and initial conditions \eq{ (p, \pt, \ptt)\vert_{t=0}=(p_0, p_1, p_2).} Above, we have set $\alpha=1-kq$ and taken
\begin{equation} \label{defBR}
q\in B_R^{\textup{W}}:=\{q\in \spaceW\, : \|q\|_{\spaceW}\leq R\mbox{ and } q(0)=p_0, \, q_t(0)=p_1, \, q_{tt}(0)=p_2
\},
\end{equation}
with some appropriately chosen radius $R>0$; recall that $\spaceW$ is defined in \eqref{regularityWest}. More precisely, crucial for our well-posedness proof will be the existence of $R>0$, such that the bounds
\begin{eqnarray}\label{smallness1}
\sqrt{C(\tau)\exp(K(\tau)(R^2+1)T)}\ r\leq R\\
\label{smallness2}
\theta=\left(\CHtwo+\CHone^2\right) C(\tau)\exp(K(\tau)(R^2+1)T) \, r |k|  \sqrt{T}<1
\end{eqnarray}
hold. By showing that $\TW$ is a self-mapping and contraction on  $B_R^\textup{W}$, we obtain the following result.
\begin{theorem}\label{Thm:West_Wellposedness}
Let assumption $(A_1)$ hold and let $\tau$, $c>0$, and $k \in \R$. Furthermore, let $r$ and $T$ be such that assumptions \eqref{smallness1} and \eqref{smallness2} hold for some $R>0$. Then for any initial data $(p_0, p_1, p_2)\in \spaceW_0$ satisfying
\begin{equation}
\nLtwo{\nabla p_2}^2 + \nLtwo{\D p_1}^2  + \nLtwo{\D p_0}^2\leq r^2
\end{equation}
and any $\delta\in[0,\bar{\delta}]$, there exists a unique solution $p\in\spaceW$ of the corresponding initial boundary-value problem for the JMGT--Westervelt equation \eqref{JMGT_West}, where $\spaceW$ is defined in \eqref{regularityWest}. The solution $p$ satisfies the estimate 
\begin{equation}
\begin{aligned}
&\nLtwo{\nabla p_{tt}(t)}^2+\nLtwo{\D p_t(t)}^2+\nLtwo{\D p(t)}^2 \\
\leq&\, C(\tau)e^{K(\tau)(R^2+1)T}
\left(\nLtwo{\nabla p_2}^2 + \nLtwo{\D p_1}^2  + \nLtwo{\D p_0}^2\right),
\end{aligned}
\end{equation}	
where the constants $C(\tau)$ and $K(\tau)$ tend to $+\infty$ as $\tau\to0^+$, but are independent of $\delta$.
\end{theorem}
\begin{proof}
The proof follows by relying on the Banach Fixed-point theorem in combination with the linear result of Proposition~\ref{propLinWest}. Indeed, by employing estimates \eqref{enest1_linwest} \textcolor{mygreen}{and \eqref{est_pttt}} with $\alpha=1-k q$, $\mu=k q_t$, $\eta=0$, $f=0$, and using  \eq{W^{1,\infty}(0,T;\Honetwo)\subseteq\spaceW,} one immediately sees that $\TW$ is a well-defined self-mapping on $B_R^{\textup{W}}$, provided \eqref{smallness1} holds. \\
\indent We next prove that $\TW$ is strictly contractive with respect to the norm on the weaker space $\textup{E}$, defined in \eqref{spaceE}. To this end, we take $q^{(1)}$ and $q^{(2)}$ in $B_R^W$ and use the short-hand notation  $p^{(1)}=\TW q^{(1)}$ and $p^{(2)}=\TW q^{(2)} $. We also introduce the differences \eq{\overline{p}=p^{(1)} -p^{(2)}}, and \eq{\overline{q}= q^{(1)} -q^{(2)}.} Then we know that $\op$ solves the linear equation
\begin{equation} \label{West_contract_eq}
\tau \op_{ttt}+(1-k q^{(1)})\op_{tt}-c^2 \Delta \op-b \Delta \op_t - k q_t^{(1)}\op_t 
= k\overline{q}_tp^{(2)}_t+k \overline{q}p^{(2)} _{tt}
\end{equation}
and has zero initial conditions. 
Estimate \eqref{enest0_linwest} with $\alpha=1-k q^{(1)}$, $\mu=k q_t^{(1)}$, $\eta=0$, and $f= k\overline{q}_tp^{(2)}_t+k \overline{q}p^{(2)} _{tt}$ yields the bound
\begin{align}
\|\op\|_{\textup{E}}
\leq&\, 
\begin{multlined}[t]\sqrt{C(\tau)\exp(K(\tau)(R^2+1)T)}|k| \Bigl(\CHtwo\nLinfLtwo{\D p^{(2)}_t}\nLtwoLtwo{\overline{q}_t}\\
+\CHone^2\nLinfLtwo{\nabla p^{(2)}_{tt}}\nLtwoLtwo{\nabla\overline{q}}\Bigr) \leq \theta \|\oq\|_{\textup{E}}.\end{multlined}
\end{align}
Above, since $\overline{q}(0)=0$ and $\overline{q}_t(0)=0$, we have exploited the bound \[\nLtwoLtwo{\overline{q}_t} +\nLtwoLtwo{\nabla\overline{q}} \leq \sqrt{T} \|\oq\|_{\textup{E}}.\] 
We also know that
\[
\nLinfLtwo{\D p^{(2)}_t}+\nLinfLtwo{\nabla p^{(2)}_{tt}}\leq\sqrt{C(\tau)\exp(K(\tau)(R^2+1)T)} \, r.
\]
Thus we obtain strict contractivity provided condition \eqref{smallness2} holds. \\
\indent  Note that the space $B_R^{\textup{W}}$ with the metric induced by the norm $\|\cdot\|_{\textup{E}}$ is a closed subset of a complete normed space. This follows from $B_R^{\textup{W}}$ being a ball and thus weakly closed in $\spaceW$. Therefore, any Cauchy sequence with respect to the norm  $\|\cdot\|_{\textup{E}}$ converges to some $y \in \textup{E}$ and has a weakly-$\star$ in $\spaceW$ convergent subsequence with limit $x \in B_R^{\textup{W}}$. Since the limits are unique, it must be $y=x \in B_R^{\textup{W}}$.\\
\indent The claim then follows by Banach's Fixed-point theorem, which at first yields a unique solution in $B_R^{\textup{W}}$. Uniqueness in $\spaceW$ follows by arguing by contradiction and employing the stability estimate on the resulting equation for the difference of two solutions; see, for example,~\cite[Theorem 2.5.1]{zheng2004nonlinear} for similar arguments.
\end{proof}
\begin{remark}[Small data or short time for the JMGT--Westervelt equation]\label{rem:smallness}
To satisfy conditions \eqref{smallness1} and \eqref{smallness2}, we can either make the radius $r$ of the initial data small or the final time $T$ short. Let  $C(\Omega,\tau,k)=(\CHtwo+\CHone^2) C(\tau)|k|$. On the one hand, for given $T$, we can read these two conditions as
\begin{equation}\label{smallnessr}
r\leq \frac{R}{\sqrt{C(\tau)\exp(K(\tau)(R^2+1)T)}}, \quad
r < \frac{1}{C(\Omega,\tau,k)\exp(K(\tau)(R^2+1)T)\sqrt{T}}
\end{equation}
for some $R>0$. They can always be satisfied by, e.g., setting $R=1$ and choosing $r>0$ small enough. Note that this smallness condition can be weakened by maximizing the right-hand sides in \eqref{smallnessr} with respect to $R$. \\
\indent On the other hand, for given $r>0$, conditions \eqref{smallness1} and \eqref{smallness2} can be satisfied by choosing short enough final time so that 
\begin{equation}\label{smallnessT}
\begin{aligned}
&T<\left(C(\Omega,\tau,k)\, r\right)^{-2}, \\
&T\leq \min\left\{1,\,\min\{\ln(R^2/(r^2C(\tau))),\,\ln(1/(rC(\Omega,\tau,k)))\}/(K(\tau(R^2+1))\right\} 
\end{aligned}
\end{equation}
for some $R>r\sqrt{C(\tau)}$. Again, the radius can simply be fixed to, for example, \eq{R=r\sqrt{C(\tau)}+1.} A more sophisticated approach would be to optimize it to make the upper bounds in \eqref{smallnessT} as large as possible. The short-time setting might be more preferable having in mind applications of nonlinear ultrasonic waves, where the data are often smooth, but not necessarily small; see, for example, \cite[\S 14.6]{kaltenbacher2007numerical} for the numerical modeling of high-intensity ultrasonic waves used in lithotripsy. 
\end{remark}

We are now ready to prove a convergence result for the nonlinear equation \eqref{JMGT_West}\textcolor{blue}{,} as the sound diffusivity vanishes.
\begin{theorem} \label{Thm:West_WeakLimit}
Under the conditions of Theorem~\ref{Thm:West_Wellposedness}, the family of solutions $\{p^{(\delta)}\}_{\delta>0}$ to the JMGT--Westervelt equation converges in the topology induced by the energy norm \eqref{energy} to a solution $p$ of the inviscid JMGT--Westervelt equation as $\delta\to0^+$ at a linear rate. 
\end{theorem}
\begin{proof}
Let $\delta$, $\delta' \in [0, \overline{\delta}]$. Again we use the short-hand notations $p^{(\delta)}$ and $p^{(\delta')}$ for the solutions of \eqref{IBVP_LinWestJMGT} with $\delta$ and $\delta'$, as well as $\op=p^{(\delta)}-p^{(\delta')}$, respectively, and prove that  $\{\op^\delta\}$ is a Cauchy sequence in the topology induced by \eqref{energy}. Note that $\op$ solves the equation
\begin{equation} \label{WestJMGT_Cauchy_eq}
\begin{aligned}
&\tau \op_{ttt}+(1-k p^{(\delta)} )\op_{tt}-(\delta+\tau c^2) \Delta \op_t-c^2 \Delta \op \\
=& \, k\overline{p}_t (p_t^{(\delta)}+p^{(\delta')}_t)+k \overline{p}p^{(\delta')}_{tt}+(\delta-\delta')\Delta p_t^{(\delta')},
\end{aligned}
\end{equation}
supplemented by zero initial conditions. 
Applying estimate \eqref{enest0_linwest} in Proposition~\ref{propLinWest} with 
$\alpha = 1-k p^{(\delta)}$, $\mu = k(p_t^{(\delta)}+p^{(\delta')}_t)$, $\eta=k p^{(\delta')}_{tt}$, and the right-hand side $f=(\delta-\delta')\Delta p_t^{(\delta')}$ directly yields
\begin{equation}
E[p](t)\leq  C(\tau)e^{K(\tau)(R^2+1)T} \, |\delta-\delta'|^2 \, R^2
\end{equation}	
with energy $E[p]$ defined as in \eqref{energy}. The stated rate is obtained by setting $\delta'$ to zero.  
\end{proof}%
\section{The linearized JMGT--Kuznetsov equation}  \label{Sec:Lin_JMGTKuzn}
We continue by investigating a linearization of the JMGT--Kuznetsov equation given by
\begin{subequations} \label{IBVP_LinKuznJMGT}
	\begin{equation}\label{LinKuznJMGT}
\tau \psi_{ttt}+\alpha\psi_{tt} - (\delta+ \tau c^2) \D \psi_t-c^2\D \psi  =\, \sigma\nabla\phi\cdot\nabla\psi_t
	\end{equation}
	\text{with homogeneous Dirichlet boundary conditions and initial conditions}
	\begin{align} \label{IC:LinKuznJMGT}
		(\psi, \psi_t, \psi_{tt})\vert_{t=0}=(\psi_0, \psi_1, \psi_2),
	\end{align}
\end{subequations}
where now the coefficient in front of the second time derivative is given by \eq{\alpha=1-\kappa \phi_t.}\\ 
\indent Since the JMGT--Kuznetsov equation has a quadratic gradient nonlinearity, we will need to obtain uniform bounds for $\nLinfLinf{\nabla \psi}$ and $\nLinfLinf{\nabla \psi_t}$ in the course of the analysis.  Our goal in this section is thus to derive a higher-order energy bound for the linearization \eqref{IBVP_LinKuznJMGT} that is uniform with respect to $\delta$ and will later allow us to derive the corresponding bound for the nonlinear equation. To this end, we strengthen our previous assumptions on the regularity of the coefficients and initial data.
\begin{itemize}
	\item[($\mathcal{A}_2$)] The coefficient $\phi$ is sufficiently smooth so that
	\begin{equation}
	\phi \in L^\infty(0,T; \Honethree) \cap W^{1, \infty}(0,T; \Honetwo),
	\end{equation}
	and uniformly bounded so that \eq{\|\phi\|_{W^{1,\infty}(\Honethree)},\, \|\phi\|_{L^{\infty}(\Honetwo)}\leq R} for some positive constant $R$, independent of $\delta$. This further implies
	\begin{equation}
	\alpha\in W^{1,\infty}(0,T;H^2(\Om))\subseteq\LinfLinf \mbox{ with } \ulal \leq \alpha \leq \olal
	\end{equation}
	for $\ulal=1-|\kappa|\CHtwo R$ and $\olal=1+|\kappa|\CHtwo R$. \vspace*{2mm}
	\item[($\mathcal{A}_3$)] The initial conditions \eqref{IC:LinKuznJMGT} satisfy
	$(\psi_0, \psi_1, \psi_2) \in \spaceK_0=\Honethree\times\Honethree\times\Honetwo.$
\end{itemize}
Note that also here we do not impose a non-degeneracy assumption on $\alpha$. 
\begin{proposition} \label{Prop:LinJMGTKuzn}
	Let $\delta \in [0, \overline{\delta}]$ and $\tau$, $c>0$. Furthermore, let assumptions \textup{$(A_1)$} and \textup{(}$\mathcal{A}_2$\textup{)}--\textup{(}$\mathcal{A}_3$\textup{)} hold. Then problem \eqref{IBVP_LinKuznJMGT} has a unique solution
	\begin{align} \label{regularityKuzn}
		\psi \in \spaceK= W^{3, \infty}(0,T; H_0^1(\Om)) \cap W^{2, \infty}(0,T; \Honetwo) \cap W^{1, \infty}(0,T; \Honethree),
	\end{align}	
which satisfies
	\begin{equation} \label{energy_est_linKuzn}
	\begin{aligned}
	&\nLtwo{\nabla \psi_{ttt}(t)}^2+\nLtwo{\D \psi_{tt}(t)}^2+\nLtwo{\nabla \D \psi_{t}(t)}^2 \\
	\leq&\, C(T, \tau, R) (\nLtwo{\D \psi_2}^2 + \nLtwo{\nabla \D \psi_1}^2  + \nLtwo{\nabla \D \psi_0}^2),
	\end{aligned}
	\end{equation}	
	where the constant $C(T, \tau, R)$ tends to $+ \infty$ as $T\to +\infty$ or $\tau\to0^+$, but is independent of $\delta$.	
\end{proposition}
\begin{proof}
	The proof can be carried out as before by employing smooth Faedo--Galerkin approximations in space, where we project the problem onto the span $V_n$ of the first $n$ eigenfunctions of the Dirichlet Laplacian pointwise in time. We will again focus our attention on deriving the crucial energy bound for the Galerkin approximations $\psi^n$. For ease of notation, we drop the superscript $n$ below. \\
	\indent We test \eqref{LinKuznJMGT} with $\D^2 \psi_{tt}$ and integrate in space. Since $\psi=\Delta \psi=0$ on $\partial \Om$ for our Galerkin approximations, the following identities hold:
	\begin{equation}
	\begin{aligned}
	\prodLtwo{\alpha \psi_{tt}}{\D^2 \psi_{tt}} =&\, 	\prodLtwo{\D [\alpha \psi_{tt}]}{\D \psi_{tt}}\\
	=&\,\prodLtwo{\alpha \D \psi_{tt}+\psi_{tt} \D \alpha +2\nabla\alpha\cdot\nabla \psi_{tt}}{\D \psi_{tt}}, \\
	-c^2\prodLtwo{\D \psi}{\D^2 \psi_{tt}} =&\, c^2\ddt \prodLtwo{\nabla \D \psi}{\nabla \D \psi_t}-c^2\nLtwo{\nabla \D \psi_t}^2.
	\end{aligned}
\end{equation}
	We thus arrive at the energy identity:
		\begin{equation} \label{LinKuzn_id}
	\begin{aligned}
	&\begin{multlined}[t]\frac12 \tau\ddt \nLtwo{\D \psi_{tt}}^2 +\frac12 (\delta+\tau c^2)\ddt \nLtwo{\nabla \D \psi_t}^2
 \end{multlined}\\
	=&\,\begin{multlined}[t] -\prodLtwo{\alpha \D \psi_{tt}+ \psi_{tt} \D \alpha +2\nabla\alpha\cdot\nabla \psi_{tt}}{\D \psi_{tt}}-c^2\ddt \prodLtwo{\nabla \D \psi}{\nabla \D \psi_t}\\+c^2\nLtwo{\nabla \D \psi_t}^2 
	+\sigma \prodLtwo{\nabla \phi \cdot \nabla \psi_t}{\Delta^2 \psi_{tt}}.\end{multlined}
	\end{aligned}
	\end{equation}
	We next integrate in time and estimate the terms arising on the right-hand side. The $\alpha$ terms can be estimated in the usual manner by utilizing H\"older's inequality and the uniform boundedness of $\alpha$ on account of assumption $(\mathcal{A}_2)$. We further have
	\[
	\begin{aligned}
	c^2\prodLtwo{\nabla \D \psi(t)}{\nabla \D \psi_t(t)} 
	\leq&\, \frac{1}{2 \varepsilon} T c^4\nLtwoLtwo{\nabla \D \psi_{t}}^2+\frac{\varepsilon}{2} \nLtwo{\nabla \D \psi_{t}(t)}^2
	\end{aligned}
	\]
where $\varepsilon \in (0, \tau c^2/2)$. Since $\nabla \phi \cdot \nabla \psi_t=0$ on $\partial \Om$ by the semi-discrete PDE, we have
\[
\begin{aligned}
\prodLtwo{\nabla \phi \cdot \nabla \psi_t}{\Delta^2 \psi_{tt}}
=&\,  \prodLtwo{\nabla \D \phi\cdot\nabla \psi_t+ 2 D^2 \phi:D^2 \psi_t+ \nabla \phi\cdot\nabla \D \psi_t}{\Delta \psi_{tt}}
\end{aligned}
\]	
where $D^2 v=(\partial_{x_i} \partial_{x_j} v)_{i,j}$ denotes the Hessian. By elliptic regularity, the Hessian satisfies \eq{\nLtwo{D^2 v}\leq C_{\textup{H}} \nLtwo{\Delta v}} 
and so we can rely on the following bound:
\[
\nLfour{D^2 v}\leq \CHone (\nLtwo{D^3 v}+\nLtwo{D^2 v})\leq \CHone C_{\textup{H}}(\nLtwo{\nabla\D v}+\nLtwo{\D v}).
\]
Thus, together with assumption $(\mathcal{A}_2)$ on the uniform boundedness of $\phi$, we have 
\[
\begin{aligned}
&\left |\int_0^t \prodLtwo{\nabla \phi \cdot \nabla \psi_t}{\Delta^2 \psi_{tt}}\ds \right|
\leq\, \begin{multlined}[t]  \nLtwoLtwo{\D \psi_{tt}} R \left\{\vphantom{\CHone^2}\nLtwoLinf{\nabla \psi_{t}}+\nLtwoLtwo{\nabla \D \psi_t}\right.\\ \left. +4 \CHone^2 C^2_{\textup{H}}(\nLtwoLtwo{\nabla \D \psi_{t}}+\nLtwoLtwo{\D \psi_{t}})\right\}.\end{multlined}
\end{aligned}
\]
Employing the derived bounds within \eqref{LinKuzn_id} after integration in time yields
\begin{equation} 
\begin{aligned}
&\begin{multlined}[t]\frac12 \tau \nLtwo{\D \psi_{tt}(t)}^2 +\frac12 (\tau c^2-\varepsilon) \nLtwo{\nabla \D \psi_t(t)}^2
\end{multlined}\\
\lesssim&\,\begin{multlined}[t]\nLtwo{\D \psi_2}^2 + \nLtwo{\nabla \D \psi_1}^2 +\nLtwo{\nabla \D \psi_0}^2 
+ R^2(\|\psi_{tt}\|_{L^2(H^2)}^2+\nLtwoLfour{\nabla \psi_{tt}}^2)\\ 
\hspace*{-0.5cm}+\nLtwoLtwo{\D \psi_{tt}}^2
+\left. R^2\nLtwoLinf{\nabla \psi_{t}}^2+(R^2+T)\nLtwoLtwo{\nabla \D \psi_t}^2\right. +R^2\nLtwoLtwo{\D \psi_{t}}^2. \end{multlined}
\end{aligned}
\end{equation}
Note that by elliptic regularity $\nLtwo{\nabla \psi_{tt}(t)} \leq \nHtwo{\psi_{tt}(t)} \leq C \nLtwo{\D \psi_{tt}(t)}.$
An application of Gronwall's inequality yields 
\begin{equation} \label{interim_bound_linKuzn}
\begin{aligned}
\nLtwo{\D \psi_{tt}(t)}^2+\nLtwo{\nabla \D \psi_{t}(t)}^2 
\leq\, C(T, \tau, R) (\nLtwo{\D \psi_2}^2 + \nLtwo{\nabla \D \psi_1}^2  + \nLtwo{\nabla \D \psi_0}^2).
\end{aligned}
\end{equation}
We can obtain an additional uniform bound on $\nLinfLtwo{\nabla \psi_{ttt}}$ by testing with $-\D \psi_{ttt}$, and relying on the assumptions on $\phi$. Standard compactness arguments allow us to carry over the result to the solution of~\eqref{IBVP_LinKuznJMGT}. Note that from $\psi \in \spaceK$, it follows that
\begin{equation}
\begin{aligned}
\psi \in\, C([0,T]; \Honethree),
\ \psi_t \in\, C_{w}([0,T]; \Honethree),\
\psi_{tt} \in\, C_{w}([0,T]; \Honetwo);
\end{aligned}
\end{equation}
cf.~\cite[Lemma 3.3]{temam2012infinite}. Thus $\psi_1$ and $\psi_2$ are attained weakly.
\end{proof}
For $f \in \LtwoLtwo$, we also briefly consider the linearization with a source term
\begin{equation}\label{LinKuznJMGT_source}
\tau \psi_{ttt}+\alpha\psi_{tt} - (\delta+ \tau c^2) \D \psi_t-c^2\D \psi  =\, \sigma\nabla\phi\cdot\nabla\psi_t+f,
\end{equation}
under the same assumptions on $\phi$. Formally testing with $\psi_{tt}$ and integrating over space and time, and employing standard computations leads to 
\begin{equation}  \label{LinKuznJMGT_source_linwest}
\begin{aligned}
&\nLtwo{\psi_{tt}(t)}^2+ \nLtwo{\nabla \psi_t(t)}^2 \\
\leq&\, C(T, \tau, R) (\nLtwo{\psi_2}^2+ \nLtwo{\nabla \psi_1}^2+\nLtwo{\nabla \psi_0}^2+ \nLtwoLtwo{f}^2),
\end{aligned}
\end{equation}
which we will rely on in the upcoming proof.

\section{Uniform bounds for the JMGT--Kuznetsov equation and the inviscid limit} \label{Sec:FixedPoint_KuznWest_Limit}
The goal of this section is to investigate the behavior of solutions to equation \eqref{JMGT_Kuzn} as $\delta \rightarrow 0^+$. As before, our work plan is to derive uniform bounds for a linearization and then relate them to the nonlinear model via a fixed-point argument. We thus introduce the mapping
\eq{\TK:\phi \mapsto \psi,} where we take $\phi$ from a suitably chosen ball in the space $\spaceK$ and $\psi$ as the solution of the linear equation \eqref{LinKuznJMGT} with initial data 
\[
(\psi(0), \psi_t(0), \psi_{tt})=(\phi(0), \phi_t(0), \phi_{tt}(0))=(\psi_0, \psi_1, \psi_2)
\] 
and $\alpha=1-\kappa \phi_t$. We next prove a small-data well-posedness result for \eqref{JMGT_Kuzn}.
\begin{theorem} \label{Thm:FixedPoint_JMGTKuzn}
Let assumption $(A_1)$ hold and let $\tau$, $c>0$ and $k \in \R$. Furthermore, let $T>0$ be given. Then there exists $r>0$, such that for any initial data $(\psi_0, \psi_1, \psi_2)\in \spaceK_0$ satisfying
\begin{equation}
\|\psi_2\|_{H^2}^2 + \|\psi_1\|_{H^3}^2  + \|\psi_0\|_{H^3}^2 \leq r^2,
\end{equation}
and any $\delta\in[0,\bar{\delta}]$, there exists a unique solution $\psi\in\spaceK$ of the corresponding initial boundary-value problem for the JMGT--Kuznetsov equation \eqref{JMGT_Kuzn}. Furthermore, the solution fulfills the estimate 
	\begin{equation}
\begin{aligned}
&\nLtwo{\nabla \psi_{ttt}(t)}^2+\nLtwo{\D \psi_{tt}(t)}^2+\nLtwo{\nabla \D \psi_{t}(t)}^2 \\
\leq&\, C(T,\tau) (\nLtwo{\D \psi_2}^2 + \nLtwo{\nabla \D \psi_1}^2  + \nLtwo{\nabla \D \psi_0}^2),
\end{aligned}
\end{equation}		
where the constant $C(T,\tau)$ tends to $+ \infty$ as $T\to\infty$ or $\tau\to0^+$, but is independent of $\delta$.	
\end{theorem}
\begin{proof}
Let $R>0$. Take $\phi \in B_R^\textup{K}$, where
\begin{equation}
\begin{aligned}
B_R^\textup{K} =\{\ \phi \in \spaceK: \ \|\phi\|_{\spaceK}\leq R,\ (\phi(0), \phi_t(0), \phi_{tt}(0))=(\psi_0, \psi_1, \psi_2)\, \}.
\end{aligned}
\end{equation}
Then \eq{\|\phi\|_{W^{1,\infty}(\Honethree)},\, \|\phi\|_{L^{\infty}(\Honetwo)}\leq R.}
Thus on account of Proposition~\ref{Prop:LinJMGTKuzn}, the mapping $\TK$ is well-defined. Furthermore, by \eqref{energy_est_linKuzn}, it is a self-mapping provided $r$ is chosen so that \eq{0<r\leq R/\sqrt{C(T, \tau, R)}.} \\
\indent To show strict contractivity in the energy norm, we take  $\phi^{(1)}, \phi^{(2)} \in B_R^\textup{K}$ and set  $\psi^{(1)}=\TK \phi^{(1)}$ and $\psi^{(2)}=\TK \phi^{(2)}$. Then $\opsi=\psi^{(1)}-\psi^{(2)}$ solves
\begin{equation} \label{Kuzn_contractivity_eq}
\begin{aligned}
&\tau \opsi_{ttt}+(1-\kappa \phi_t^{(1)} )\opsi_{tt}-(\delta+\tau c^2) \Delta \opsi_t-c^2 \Delta \opsi-\sigma \nabla \phi^{(2)} \cdot \nabla \opsi_t \\
=&\, \kappa \overline{\phi}_t \psi^{(2)}_{tt}+\sigma \nabla \overline{\phi} \cdot \nabla \psi^{(1)}_t 
\end{aligned}
\end{equation}
and satisfies zero initial conditions.  This equation corresponds to \eqref{LinKuznJMGT_source} if we set 
\[
f= \kappa \overline{\phi}_t \psi^{(2)}_{tt}+\sigma \nabla \overline{\phi} \cdot \nabla \psi^{(1)}_t 
\]
together with $\alpha=1-\kappa \phi_t^{(1)}$. Using estimate \eqref{LinKuznJMGT_source_linwest} thus yields
\begin{equation}  \label{Contractivity_identity}
\begin{aligned}
\nLtwo{\opsi_{tt}(t)}^2+ \nLtwo{\nabla \opsi_t(t)}^2+\nLtwo{\nabla \opsi(t)}^2 
\leq C(T, \tau, R) \nLtwoLtwo{f}^2
\end{aligned}
\end{equation}
and then it remains to estimate the $f$ term. By H\"older's inequality, we have
\begin{equation}
\begin{aligned}
\nLtwoLtwo{f}
\leq&\,\begin{multlined}[t] |\kappa| \nLinfLfour{\psi_{tt}^{(2)}}\nLtwoLfour{\overline{\phi}_t}  +|\sigma| \nLinfLinf{\nabla \psi_t^{(1)}}\nLtwoLtwo{\nabla \overline{\phi} }.
\end{multlined}
\end{aligned}
\end{equation}
Furthermore,
$ \nLinfLfour{\psi_{tt}^{(2)}}\nLtwoLfour{\overline{\phi}_t}
\leq\,  \CHone^2 \nLinfLtwo{\nabla \psi_{tt}^{(2)}}T\nLinfLtwo{\nabla \overline{\phi}_t}.$
By additionally noting that $\nLtwoLtwo{\nabla \overline{\phi} } \leq T \nLtwoLtwo{\nabla \overline{\phi}_t}$, it further follows that
\begin{equation}
\begin{aligned}
\nLtwoLtwo{f}^2
 \lesssim&\,\begin{multlined}[t] 2\kappa^2\CHone^4 \nLinfLtwo{\nabla \psi_{tt}^{(2)}}^2T^2\nLinfLtwo{\nabla \overline{\phi}_t}^2\\ \hspace*{2cm}
+2\sigma^2 \nLinfLinf{\nabla \psi_t^{(1)}}^2T^2\nLtwoLtwo{\nabla \overline{\phi_t} }^2. 
\end{multlined}
\end{aligned}
\end{equation}
Employing this bound in \eqref{Contractivity_identity} and relying on Gronwall's inequality leads to
\begin{equation} 
\begin{aligned}
&\sup_{t \in (0,T)}\nLtwo{\opsi_{tt}(t)}+\sup_{t \in (0,T)}\nLtwo{\nabla \opsi_t(t)} \\
\leq&\,\begin{multlined}[t] C(T, R)(\nLinfLtwo{\nabla \psi_{tt}^{(2)}}+\nLinfLinf{\nabla \psi_t^{(1)}})T\sup_{t \in (0,T)}\nLtwo{\nabla \overline{\phi}_t(t)}. \end{multlined}
\end{aligned}
\end{equation}
Note that
$\displaystyle \sup_{t \in (0,T)} \nLtwo{\nabla \overline{\phi}_t(t)} \leq \|\overline{\phi}\|_{\textup{E}}$. By \eqref{energy_est_linKuzn}, we know that
\[
\nLinfLtwo{\nabla \psi_{tt}^{(2)}}+\nLinfLinf{\nabla \psi_t^{(1)}} \leq \sqrt{\tilde{C}(T, \tau, R)}\, r
\]
for some $\tilde{C}(T, \tau, R)>0$, independent of $\delta$. Thus we can achieve strict contractivity of $\TK$ in the energy norm by reducing $r$. We can reason as in the proof of Theorem~\ref{Thm:West_Wellposedness} concerning $B_R^\textup{K}$ being closed with respect to the metric induced by $\|\cdot\|_{\textup{E}}$ and establish our claim by employing the Banach Fixed-point theorem.
\end{proof}
\begin{remark}[Small data or short time for the JMGT--Kuznetsov equation]
Like in the JMGT--Westervelt case (cf. Remark~\ref{rem:smallness}), instead of choosing the maximal magnitude of the initial data $r$ small enough for fixed final time $T$, we could have also achieved the self-mapping and contractivity properties needed for proving Theorem~\ref{Thm:FixedPoint_JMGTKuzn} by choosing $T$ small enough, given initial data that is smooth but of arbitrary size. 
\end{remark}
We conclude our theoretical investigations by proving a convergence result for the JMGT--Kuznetsov equation \eqref{JMGT_Kuzn}. Note that as a by-product, we also obtain a convergence result for the JMGT--Westervelt equation in potential form~\eqref{JMGT_West_potential} by setting $\kappa=\frac{1}{c^2}\left(1+\frac{B}{2A}\right) $ and $\sigma=0$ in ~\eqref{JMGT_Kuzn} .
\begin{theorem} \label{Thm:Kuzn_Limit}
Let the assumptions of Theorem~\ref{Thm:FixedPoint_JMGTKuzn} hold. Then the family of solutions $\{\psi^{(\delta)}\}_{\delta>0}$ to the JMGT--Kuznetsov equation \eqref{JMGT_Kuzn} converges  in the topology induced by the energy norm for the wave equation at a linear rate to the solution $\psi$ of the inviscid JMGT--Kuznetsov equation as $\delta\to0^+$. 
\end{theorem}	
\begin{proof}
Let $\delta$, $\delta' \in [0, \overline{\delta}]$. The difference equation for $\opsi=\psi^{(\delta)}-\psi^{(\delta')}$ is given by
\begin{equation} \label{KuznJMGT_Cauchy_eq}
\begin{aligned}
&\tau \opsi_{ttt}+(1-\kappa \psi_t^{(\delta)} )\opsi_{tt}-(\delta+\tau c^2) \Delta \opsi_t-c^2 \Delta \opsi \\
=& \, \kappa \overline{\psi}_t\psi^{(\delta')} _{tt}+(\delta-\delta')\Delta \psi_t^{(\delta')}+\sigma \nabla \opsi \cdot \nabla \psi_t^{(\delta)}+\sigma \nabla \psi^{(\delta')} \cdot \nabla \opsi_t
\end{aligned}
\end{equation}
with $(\opsi, \opsi_t, \opsi_{tt})\vert_{t=0}=(0, 0, 0)$. Testing this equation with $\opsi_{tt}$ and integrating over space and time leads to
\begin{equation} \label{Limit_identity}
\begin{aligned}
&\frac{\tau}{2}\nLtwo{\opsi_{tt}(t)}^2+\nLtwoLtwo{\sqrt{\alpha}\,\opsi_{tt}}^2+\frac{\delta+\tau c^2}{2}\nLtwo{\nabla \opsi_t(t)}^2 \\
=&\,\begin{multlined}[t] \kappa \intTO  \psi_{tt}^{(\delta')} \opsi_t \opsi_{tt} \dxs+\sigma \intTO ( \nabla \opsi \cdot \nabla \psi_t^{(\delta)}+ \nabla \psi^{(\delta')} \cdot \nabla \opsi_t) \opsi_{tt}\dxs\\
+ (\delta-\delta')\intTO  \D \psi_t^{(\delta')} \opsi_{tt} \dxs-c^2 \intTO  \nabla \opsi \cdot \nabla \opsi_{tt} \dxs, \end{multlined}
\end{aligned}
\end{equation}
where now $\alpha=1-\kappa \psi_t^{(\delta)}$. We can proceed similarly to the proof of contractivity in Theorem~\ref{Thm:FixedPoint_JMGTKuzn} to arrive at the estimate
\begin{equation}
\sup_{t \in (0,T)} E[\opsi](t) \lesssim (\delta-\delta')^2\nLtwoLtwo{\D \psi_t^{(\delta')}}^2.
\end{equation}
Setting $\delta'$ to zero yields the claimed convergence rate.
\end{proof}
\begin{remark}[Different boundary conditions] It should be said that, to model ultrasonic propagation in realistic settings, acoustic equations are in practice often considered with Neumann excitation and possibly absorbing boundary conditions to avoid non-physical reflections; see, e.g.,~\cite[\S 5]{kaltenbacher2007numerical}.  We expect our analysis to carry over in a relatively straightforward manner to problems with Neumann conditions, under suitable assumptions on the data. The case of having both Neumann and absorbing conditions on disjoint parts of the boundary is more delicate due to the fact that our energy analysis would imply only $H^{3/2+s}$ regularity in space of the solution for  $s \in [0, 1/2)$. We refer the reader to~\cite[\S 6]{kaltenbacher2019vanishing}, where the analysis of the JMGT--Westervelt equation is performed with Neumann and the lowest-order Engquist--Majda boundary conditions for $\delta>0$ fixed and vanishing thermal relaxation time. We would thus expect that the arguments of~\cite[\S 6]{kaltenbacher2019vanishing} can be adapted to the present setting of vanishing sound diffusivity for the generalized MGT and JMGT--Westervelt equations, whereas it does not seem immediately feasible to extend these considerations to the JMGT--Kuznetsov model.
\end{remark}	
\vspace*{-2mm}

\section{Numerical results} \label{Sec:NumResults}
In this section, we illustrate some of our previous theoretical results numerically by employing a Matlab implementation. For discretization in the spatial variable, we use continuous piecewise linear finite elements on a uniform discretization of the computational domain with mesh size $h$. In time, we rely on a Newmark discretization for third-order in time equations  realized as a predictor-corrector scheme; we refer to~\cite[\S 8]{KaltenbacherNikolic} for details. The three parameters within the scheme are chosen as $(1/12,1/4,1/2)$. Having in mind our discussion  in Remark~\ref{Remark:PerturbedSpeed} regarding the perturbation of the speed of sound, we heuristically choose the time step $\Delta t$ so that
\begin{align} \label{CFL_condition}
\left(c+\sqrt{\overline{\delta}/\tau}\right)\Delta t \leq \textup{CFL}\cdot h
\end{align}
with $\textup{CFL}=0.1$. The nonlinearity is resolved via a fixed-point iteration, where we treat the whole nonlinear term as the previous iterate and set the tolerance to $\textup{TOL}=10^{-8}$. \\
\indent We consider sound propagation through water, where the speed of sound is taken to be $c=\SI{1500}{m/s}$, density $\varrho=\SI{1000}{kg/m^3}$, and the parameter of nonlinearity is set to $B/A=5$; cf.~\cite[\S 5]{kaltenbacher2007numerical} and~\cite{beyer1998parameter}. For sea water, at least two  molecular relaxation processes are known to be pronounced, with molecular relaxation times $\tau_1=\SI{1e-3}{s}$ and $\tau_2=\SI{1.5e-5}{s}$; see~\cite[\S 1]{naugolnykh2000nonlinear}. We choose to adopt the same values for the thermal relaxation time in our numerical experiments and additionally test with $\tau=\tau_0=\SI{1.5e-7}{s}$ to observe the effects of the dissipation and nonlinearity when $\tau$ is relatively small.
\subsection{Nonlinear propagation in a channel}
We first consider nonlinear propagation in a narrow channel, as modeled by the JMGT--Westervelt equation \eqref{JMGT_West} with $p=p(x,t)$ and $x \in \Omega=(0, 0.4)$. We set initial data to
\[
(p, p_t, p_{tt})\vert_{t=0}=\left(\mathcal{A}\exp\left(-\frac{(x-0.2)^2}{2\sigma^2}\right), 0, 0 \right),
\]
with $\mathcal{A}=\SI{100}{MPa}$ and $\sigma=0.01$, which corresponds to a Gaussian-like initial pressure distribution centered around $x=0.2$. The sound diffusivity $\delta$ is expected to be relatively small in water with values in the interval $[10^{-9}, 10^{-4}] \, \textup{m}^2/\textup{s}$; cf.~\cite[\S 5]{kaltenbacher2007numerical} and~\cite[\S 5]{walsh2007finite}. To demonstrate the influence of this parameter on the solutions of the equation, we test the problem in an exaggerated setting by choosing $\delta \in \{0, 1\}\, \textup{m}^2/\textup{s}$. To resolve the nonlinear behavior, we employ $600$ elements in space and choose the time step according to \eqref{CFL_condition}.  \\
\indent Figures~\ref{Fig:Sensitivity1}--\ref{Fig:Sensitivity3} depict the acoustic pressure at final time $T=\SI{70}{\mu s}$ for two different values of $\delta$. The difference between the linear and nonlinear pressure distribution is also displayed in thermally relaxing, inviscid media (where $\delta=0$) on the right. In Figure~\ref{Fig:Sensitivity1}, the thermal relaxation parameter is taken to be relatively small, $\tau=\SI{1.5e-7}{s}$. Thus, we expect the behavior as observed in the corresponding second-order model. Indeed, we see that increasing $\delta$ leads to the damping of the amplitude and subduing the nonlinear steepening of the wavefront. For a rigorous study into the behavior of third-order acoustic models as $\tau \rightarrow 0^+$ with $\delta>0$ fixed, we refer to, for example,~\cite{KaltenbacherNikolic, bongarti2020vanishing}. 
\begin{figure}[h]
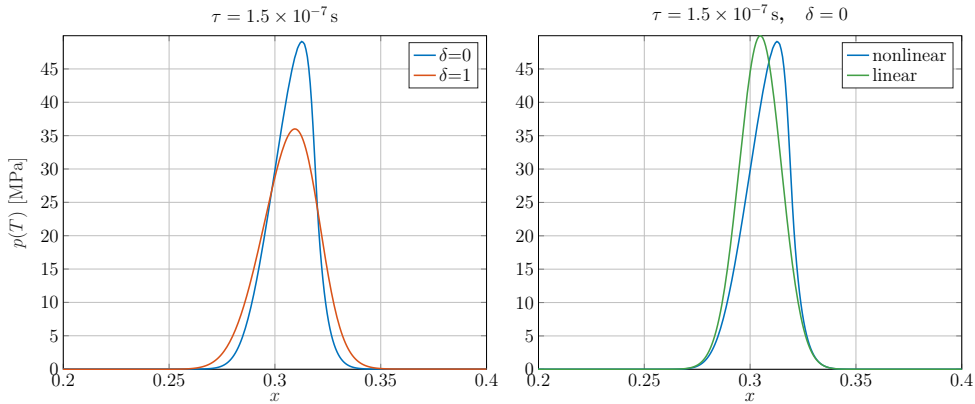

	\hspace*{-0.3cm}	
	\scalebox{0.49}{\input{plot_u1}}\hspace*{0.0cm}
	\scalebox{0.49}{\input{lin_vs_nonlin1}}
	
	\caption{\textbf{(left)} Pressure at final time for different values of $\delta$ and fixed, small $\tau$ \textbf{(right)} Linear and nonlinear pressure distribution at final time with $\delta=0$}
	\label{Fig:Sensitivity1}
	
\end{figure}

\indent In Figure~\ref{Fig:Sensitivity2}, the thermal relaxation time is set to $\SI{1.5e-5}{s}$, which appears to be in an intermediate range in terms of the displayed effects. Here increasing $\delta$ influences the amplitude of the wave and its propagation speed. We also note that in this parameter regime, the nonlinear effects are subdued compared to the case of having a shorter relaxation time. \\  
\indent Figure~\ref{Fig:Sensitivity3} displays the pressure distribution when the thermal relaxation time is relatively large, $\tau=\SI{1e-3}{s}$. Here the effects of thermal relaxation appear to overtake both the effects of dissipation and nonlinearity. In fact, here it might be more sensible to observe $z=\tau u_t+u$, which we also plot at final time in Figure~\ref{Fig:Sensitivity4}. We see that the effects of increasing $\delta$ are practically negligible. Our parameter study in Figures~\ref{Fig:Sensitivity1}--\ref{Fig:Sensitivity4} suggests that a deeper theoretical investigation into the interplay among $\delta$, $\tau$, and the nonlinear parameters is of interest.
\begin{figure}[h]
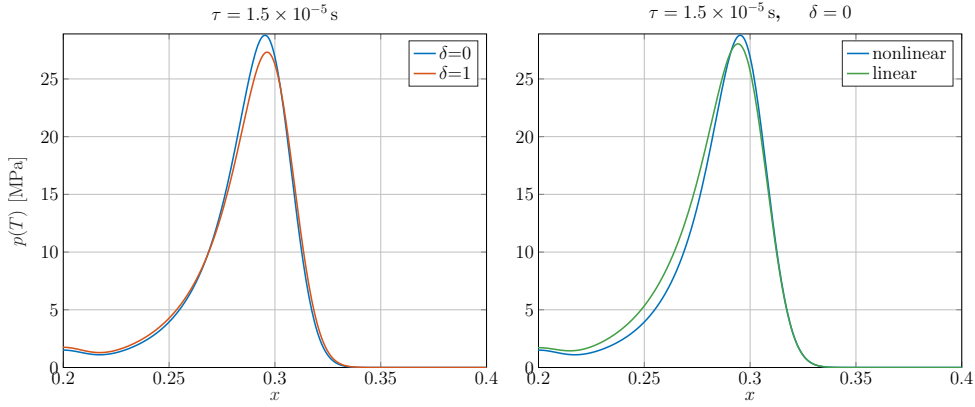

	\hspace*{-0.3cm} \scalebox{0.49}{\input{plot_u2}}\hspace*{0.0cm}
	\scalebox{0.49}{\input{lin_vs_nonlin2}}
	
	\caption{\textbf{(left)} Pressure at final time for different values of $\delta$ and medium $\tau$ \textbf{(right)}
		Linear and nonlinear pressure distribution at final time with $\delta=0$}
	\label{Fig:Sensitivity2}	
\end{figure}%
\begin{figure}[h]
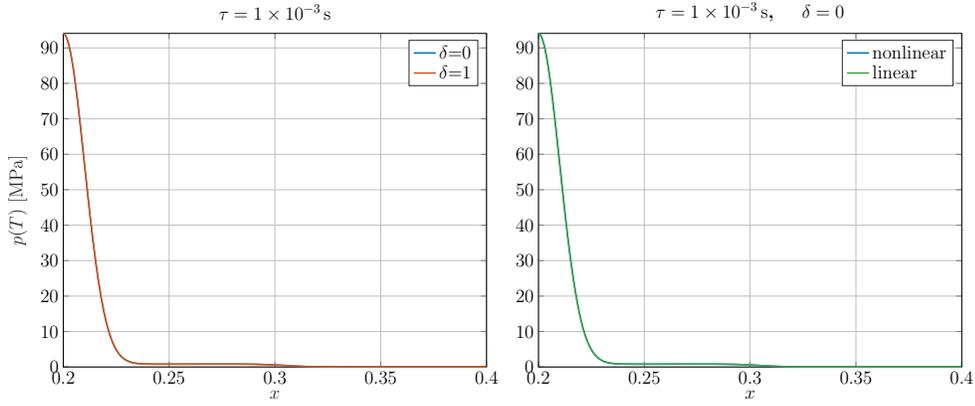

	\hspace*{-0.3cm}	
	\scalebox{0.49}{\input{plot_u3}}\hspace*{0.0cm}
	\scalebox{0.49}{\input{lin_vs_nonlin3}}
	\caption{\textbf{(left)} Pressure at final time for different values of $\delta$ and large
		$\tau$ \textbf{(right)} Linear and nonlinear pressure distribution at final time with $\delta=0$}
	\label{Fig:Sensitivity3}
	
\end{figure}

\indent For the convergence study, we take $\delta \in [0, 10^{-2}]\, \textup{m}^2/\textup{s}$. In the experiments we conducted the same rate of convergence was obtained for the three values of the thermal relaxation parameter considered before; we thus present here only the case $\tau=\SI{1.5e-5}{s}$. The plot of the relative error in the energy norm 
\begin{align} \label{approx_error}
\textup{err}=\frac{\|p_h^{(\delta)}-p_h\|_{\textup{E}}}{\|p_h\|_{\textup{E}}}
\end{align}
is given in Figure~\ref{Fig:Convergence1D2D} on the left, where
\[
\|p_h\|_{\textup{E}}=\sup_{t \in (0,T)} \nLtwo{p_{h, tt}(t)}+\sup_{t \in (0,T)}\nLtwo{\nabla p_{h,t}(t)}.
\]
We observe a linear rate of convergence, as expected on account of Theorem~\ref{Thm:West_WeakLimit}.
 \begin{figure}[h]
	\centering
	\scalebox{0.49}{\input{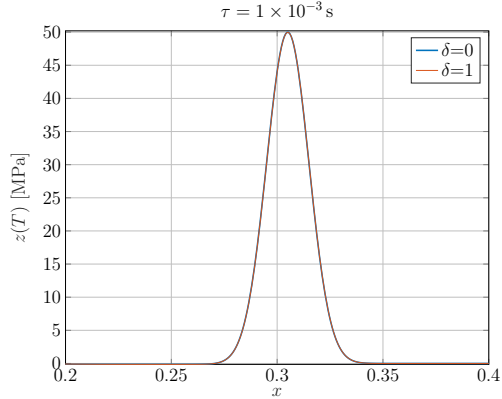}}
	\caption{Distribution of $z=\tau p_t+p$ at final time for two different values of $\delta$ and a relatively large thermal relaxation time $\tau$ }
	\label{Fig:Sensitivity4}
\end{figure}
\begin{figure}[ht]
	\hspace*{-0.3cm}
	\scalebox{0.45}{
%
\definecolor{mycolor1}{rgb}{0.00000,0.44700,0.74100}%
\begin{tikzpicture}[font=\Large]

\begin{axis}[%
width=4.521in,
height=3.566in,
at={(0.758in,0.481in)},
scale only axis,
xmin=0,
xmax=0.01,
xlabel style={font=\color{white!15!black}},
xlabel={\Large $\delta$},
ymin=0,
ymax=0.00190908136705375,
ylabel style={font=\color{white!15!black}},
ylabel={\Large Relative error},
axis background/.style={fill=white},
title style={font=\bfseries},
xmajorgrids,
ymajorgrids,
legend style={legend cell align=left, align=left, draw=white!15!black}
]
\addplot [color=mycolor1, line width=1.2pt]
table[row sep=crcr]{%
0	0\\
1e-10	1.91072979490226e-11\\
1e-09	1.90995227506312e-10\\
1e-08	1.90989922417742e-09\\
1e-07	1.90988938287255e-08\\
1e-06	1.9098878698146e-07\\
1e-05	1.90988704083639e-06\\
0.0001	1.9098797837469e-05\\
0.001	0.000190980721004107\\
0.01	0.00190908136705375\\
};

\end{axis}
\end{tikzpicture}%
}\hspace*{0.0cm}
	\scalebox{0.45}{
%
\definecolor{mycolor1}{rgb}{0.00000,0.44700,0.74100}%
\begin{tikzpicture}[font=\Large]

\begin{axis}[%
width=4.521in,
height=3.566in,
at={(0.758in,0.481in)},
scale only axis,
xmin=0,
xmax=0.01,
xlabel style={font=\color{white!15!black}},
xlabel={\Large $\delta$},
ymin=0,
ymax=0.001324967537783,
ylabel style={font=\color{white!15!black}},
ylabel={\Large Relative error},
xmajorgrids,
ymajorgrids,
axis background/.style={fill=white},
legend style={legend cell align=left, align=left, draw=white!15!black}
]
\addplot [color=mycolor1, line width=1.2pt]
  table[row sep=crcr]{%
0	0\\
1e-10	1.32560447369846e-11\\
1e-09	1.32555158979924e-10\\
1e-08	1.32555755890381e-09\\
1e-07	1.32555759377174e-08\\
1e-06	1.32555750399122e-07\\
1e-05	1.32555697863721e-06\\
0.0001	1.32555166724341e-05\\
0.001	0.000132549855745382\\
0.01	0.001324967537783\\
};

\end{axis}
\end{tikzpicture}%
}
	\caption{ Relative error in the energy norm with respect to $\delta$ ({\bf left}) Nonlinear propagation in a channel ({\bf right}) Two-dimensional linear propagation with an external source of sound}
	\label{Fig:Convergence1D2D}	
\end{figure}
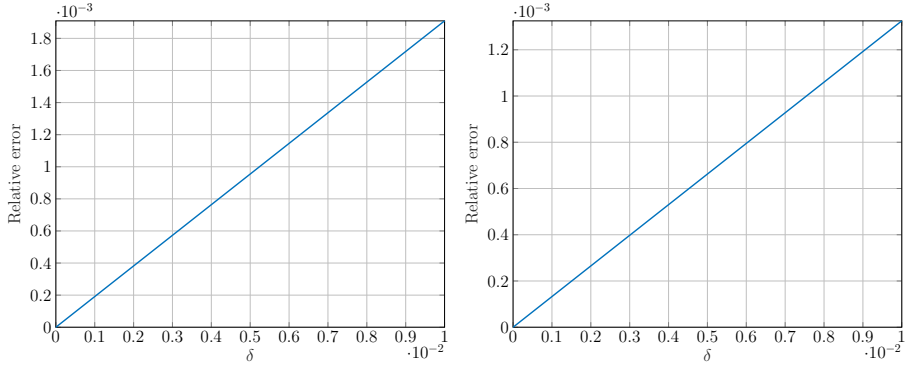
\subsection{Linear propagation with an external source} We also illustrate our convergence results for the linear MGT equation \eqref{MGT} in a two-dimensional setting with $\alpha=1$ and a source term. For $\Omega=(0, 0.5) \times (0, 0.5)$, we take the source term to be
\begin{equation}
\begin{aligned}
f(x,y, t)= \mathcal{A}\exp\left(-\frac{(x-x_0)^2}{2\sigma_x^2}-\frac{(y-y_0)^2}{2\sigma_y^2}\right)\sin(w t),
\end{aligned}
\end{equation}
where $\mathcal{A}=10^{10}$, $x_0=y_0=0.25$, $\sigma_x=0.02$, and $\sigma_y=0.01$. Furthermore, the frequency is set to $w=2\pi f$ with $f=2 \cdot 10^4$. The initial conditions $(p_0, p_1, p_2)$ are assumed to be zero. We employ a uniform triangular mesh with mesh size $h=0.01$ to discretize the computational domain. The time step is then chosen according to \eqref{CFL_condition}. We take again the medium parameters of water and set the thermal relaxation time to $\tau=\SI{1.5e-5}{s}$. Figure~\ref{Fig:2DSnapshots} provides snapshots of the approximate acoustic pressure in inviscid media, where $\delta=0$, until the final time $T=\SI{1.5e-4}{s}$. 
\begin{figure}[h]
	\subcaptionbox{$t=\SI{3.8e-5}{s}$}{\includegraphics[scale=0.46, trim={0 0 0.2cm 0}]{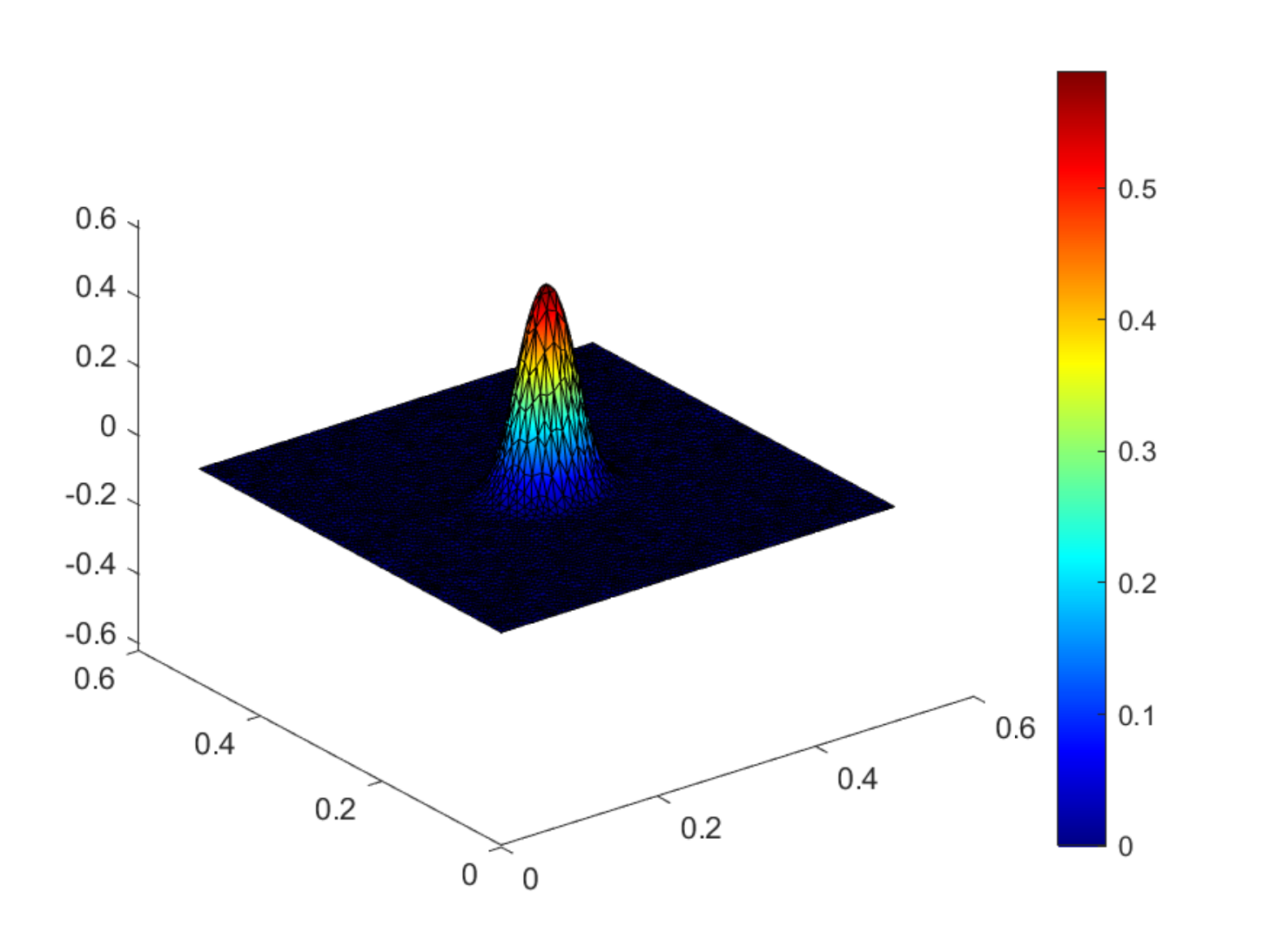}}\hspace*{-0.3cm}
	\subcaptionbox{$t=\SI{4.8e-5}{s}$}{\includegraphics[scale=0.46, trim={0 0 0.7cm 0}]{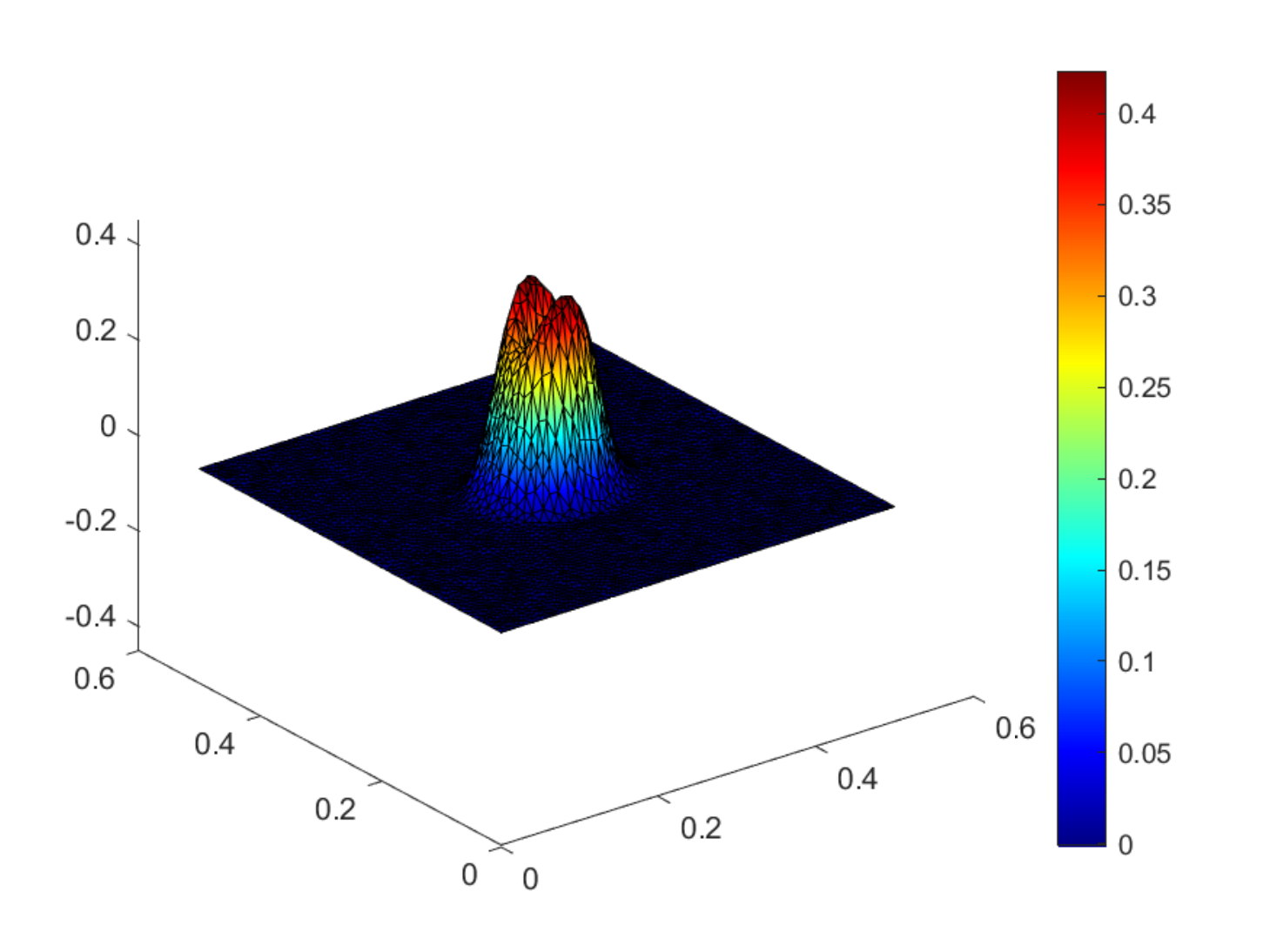}}\\
	\subcaptionbox{$t=\SI{8.7e-5}{s}$}{\includegraphics[scale=0.46, trim={0 0 0.2cm 0}]{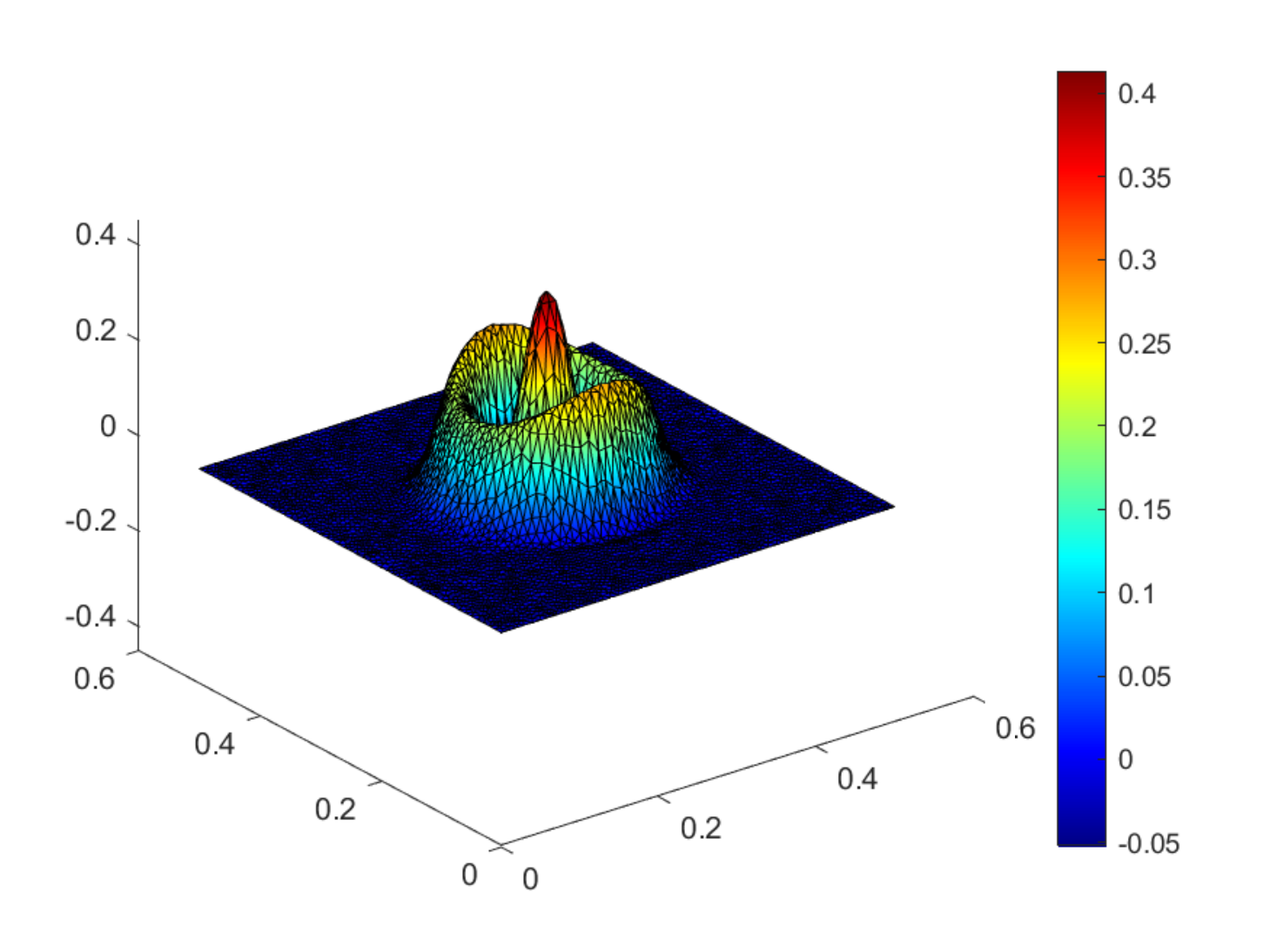}}\hspace*{-0.3cm}
	\subcaptionbox{$t=T$}{\includegraphics[scale=0.46, trim={0 0 0.7cm 0}]{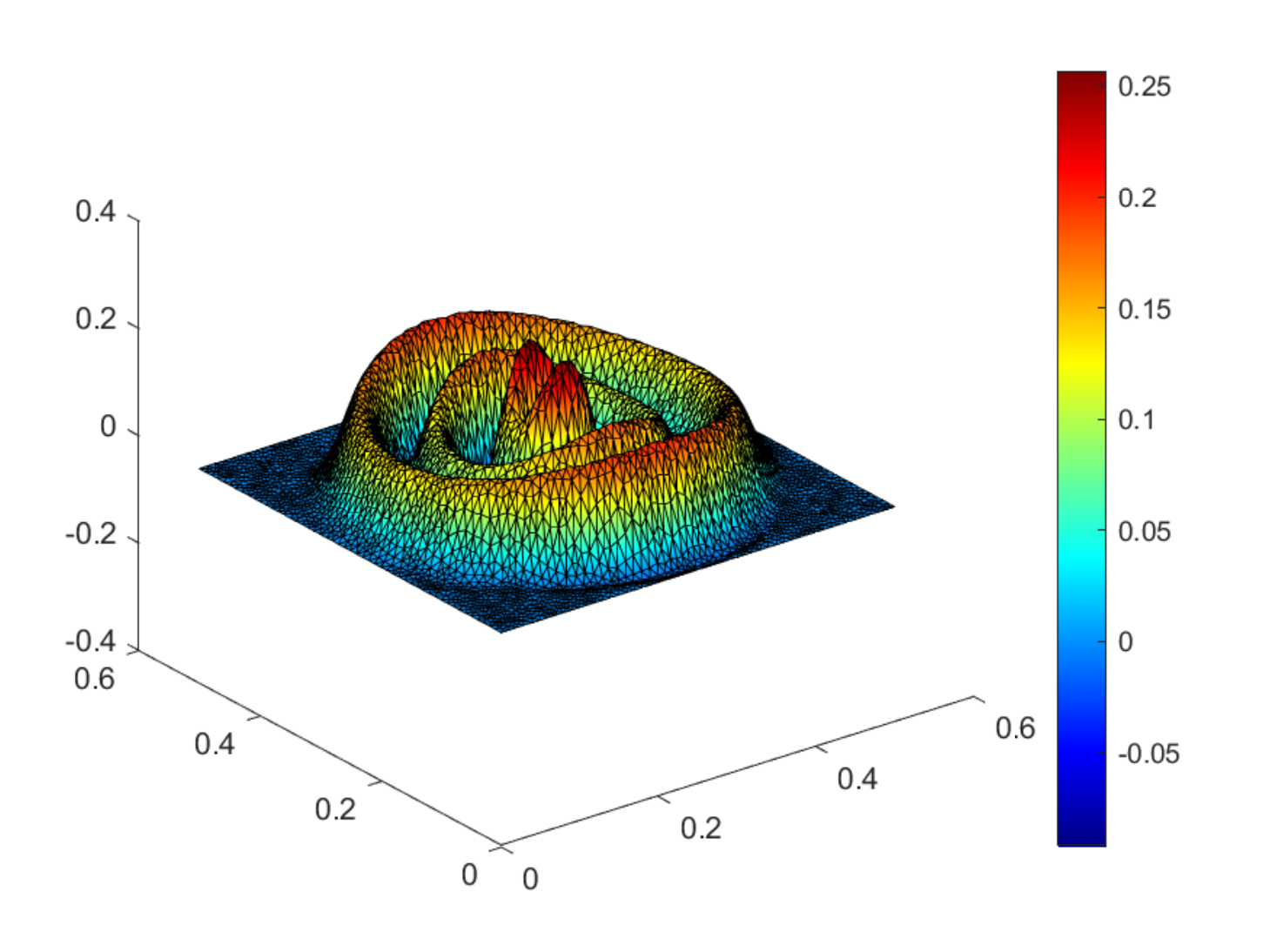}}
	\caption{Linear evolution of the acoustic pressure over time in thermally relaxing, inviscid media in the presence of an external source of sound}  	\label{Fig:2DSnapshots} 	
\end{figure}

\indent To perform the convergence study, we take $\delta \in [0, 10^{-2}]\, \textup{m}^2/\textup{s}$ and compute the relative error in the energy norm according to \eqref{approx_error}. The plot is given in Figure~\ref{Fig:Convergence1D2D} on the right. We observe again a linear convergence rate with respect to $\delta$, as we expected based on the result of Theorem~\ref{Thm:West_WeakLimit_lin}.%

\begin{remark}[Propagation through tissue-like media] 
In different biomedical applications, including lithotripsy, ultrasonic waves propagate through both fluidic and tissue-like (heterogeneous) media. It is well-known that in tissues   , the attenuation of the wave obeys a power law with respect to the frequency; we refer to the book~\cite{holm2019waves} for a detailed insight into these phenomena. These effects are typically modeled by space- or time-fractional wave equations; see, for example,~\cite[\S 7, Eq. (7.8)]{holm2019waves}. It is thus of clear interest to involve fractional propagation in the future analytical and numerical investigations. 
\end{remark}	
\section*{Discussion and outlook}
In this paper, we have investigated the limiting behavior of the third-order JMGT equations and their linearizations, as the sound diffusivity vanishes. The analysis also included well-posedness of the respective limiting equations as well as uniform $\delta$-independent energy estimates. From our estimates and even more clearly from our numerical experiments, it is apparent that there is a rather involved interplay among the diffusivity parameter $\delta$, the relaxation time $\tau$, and the nonlinearity (determined by the parameters $\kappa$ and $\sigma$). Analytic and further numerical studies of this interaction will, therefore, be the subject of further research. 
\end{document}